\documentclass[11pt]{amsart}
\usepackage{amsmath,amssymb,color}
\usepackage{amsthm}
\usepackage{hyperref}
\usepackage[margin=1.0in]{geometry}
\usepackage{cleveref}
\usepackage[cm]{fullpage}
\usepackage{enumitem}
\usepackage{amscd}
\usepackage{ esint }
\usepackage{systeme,mathtools}
\usepackage{amsfonts}
\usepackage{orcidlink}
\definecolor{orcidlogocol}{HTML}{A6CE39}
\usepackage{amsthm}
\usepackage{tikz-cd}
\usepackage{stmaryrd}
\usepackage{lipsum}
\usepackage{appendix}
\usepackage{hyperref}
\usepackage{cleveref}
\DeclareMathOperator{\divg}{div}
\DeclareMathOperator{\im}{Im}
\DeclareMathOperator{\tr}{tr}
\DeclareMathOperator{\Aut}{Aut}
\DeclareMathOperator{\End}{End}
\DeclareMathOperator{\Diff}{Diff}
\DeclareMathOperator{\GDiff}{GDiff}
\DeclareMathOperator{\Isom}{Isom}
\DeclareMathOperator{\Rc}{Rc}

\begin{document}

\theoremstyle{definition}
\newtheorem{claim}{Claim}
\theoremstyle{plain}
\newtheorem{proposition}{Proposition}[section]
\newtheorem{theorem}[proposition]{Theorem}
\newtheorem{lemma}[proposition]{Lemma}
\newtheorem{corollary}[proposition]{Corollary}
\theoremstyle{definition}
\newtheorem{defn}[proposition]{Definition}
\theoremstyle{remark}
\newtheorem{remark}[proposition]{Remark}
\theoremstyle{definition}
\newtheorem{example}[proposition]{Example}
\theoremstyle{definition}
\newtheorem*{Motivation}{Motivation}

\title{Dynamical stability of Pluriclosed and Generalized Ricci solitons}

 \author{Kuan-Hui Lee}
 \address{Rowland Hall\\
          University of California, Irvine\\
          Irvine, CA 92617}
 \email{\href{mailto:kuanhuil@uci.edu}{kuanhuil@uci.edu}}

\begin{abstract}
In this work, we discuss the stability of the pluriclosed flow and generalized Ricci flow. We proved that if the second variation of generalized Einstein--Hilbert functional is nonpositive and the infinitesimal deformations are integrable, the flow is dynamically stable. Moreover, we prove that the pluriclosed steady solitons are dynamically stable when the first Chern class vanishes. 
\end{abstract}

\maketitle

\section{Introduction}

Let $(M^{2n},J)$ be a complex manifold and $\omega$ be a Hermitian metric on $M$. We say that the metric $\omega$ is pluriclosed if $\partial \overline{\partial}\omega=0$. In this work, we consider the evolution equation for $\omega$ which is given by
\begin{align*}
    \frac{\partial}{\partial t}\omega =-\rho_B^{1,1},
\end{align*}
where $\rho_B$ denotes the Bismut Ricci curvature (\ref{BRF}). This evolution equation was first introduced in \cite{J3}, motivated by the Kähler–Ricci flow. It is well-known that the Ricci flow preserves Kähler geometry, and Cao initiated the study of the Kähler–Ricci flow in \cite{Cao1985}, where he used it to reprove the Calabi–Yau theorem \cite{Yau} and the Aubin–Yau theorem \cite{Aubin,Yau}.
 However, there exist many examples of complex manifolds that are non-Kähler, starting with the classical example of Hopf surface $S^3\times S^1$. The Kähler–Ricci flow cannot be applied to such manifolds, and it can be shown that on non-Kähler manifolds, the Ricci flow does not even preserve the Hermitian condition. Therefore, one must look elsewhere to adapt the ideas of geometric flows to this setting. The pluriclosed flow was introduced as a natural extension of the Kähler–Ricci flow that preserves the Hermitian condition for non-Kähler metrics.

Rather than extending the Ricci flow within complex geometry directly, one can also consider the Ricci flow through connections that are not the Levi-Civita connection. Let 
$M$ be a smooth manifold equipped with a closed 3-form $H_0$. A one-parameter family of Riemannian metrics $g_t$  and 2-forms $b_t$ is said to be a solution of the generalized Ricci flow if it satisfies
\begin{align*}
   &\frac{\partial}{\partial t}g=-2\Rc+\frac{1}{2}H^2,\nonumber
   \\&\frac{\partial}{\partial t}b=-d^*H \quad \text{ where $H=H_0+db$.} 
\end{align*}
This parabolic system was formulated in \cite{P, J1}, and it can be interpreted as the Ricci flow associated with the Bismut connections $\nabla^{\pm}=\nabla\pm\frac{1}{2}g^{-1}H$. The generalized Ricci flow arises naturally in generalized geometry, a framework inspired by Poisson geometry \cite{MG, STREETS2017506}, complex geometry \cite{J7, J3}, generalized geometry \cite{hitchin}, and mathematical physics \cite{polchinski_1998}. Furthermore, it is known that the generalized Ricci flow is gauge equivalent to the pluriclosed flow.

We define the generalized Einstein--Hilbert functional
\begin{align*}
    \mathcal{F}\colon\nonumber&\quad \Gamma(S^2M)\times \Omega^2\times C^\infty (M)\to \mathbb{R}
    \\& \quad (g,b,f)\longmapsto \int_M (R-\frac{1}{12}|H_0+db|^2+|\nabla f|^2)e^{-f}dV_g 
\end{align*}
and  
\begin{align*}
    \lambda(g,b)\coloneqq\inf\Big\{\mathcal{F}(g,b,f)\big|\kern0.2em f\in C^\infty(M),\,\int_M e^{-f}dV_g=1\Big\}. 
\end{align*}
One can see that $\lambda(g,b)$ can be achieved by some $f$ uniquely, i.e., $\lambda(g,b)=\mathcal{F}(g,b,f)$ and $\lambda$ is the first eigenvalue of the Schrödinger operator $-4\triangle+R-\frac{1}{12}|H_0+db|^2$. In \cite{P}, it was shown that $\lambda$ is monotone increasing under the generalized Ricci flow and critical points of $\lambda$ are steady gradient generalized Ricci solitons. More precisely, we say that $(g,b)$ is a steady gradient generalized Ricci soliton if it satisfies
\begin{align*}
     0=\Rc-\frac{1}{4}H^2+\nabla^2 f, \quad 0=d_g^*H+i_{\nabla f}H, \quad \text{where } H=H_0+db. 
\end{align*}

One of the main goals in this paper is to study the stability of steady gradient generalized Ricci solitons. Let $(g,b)$ be a steady gradient generalized Ricci soliton on a smooth manifold $M$. Suppose $(g_t,b_t)$ is a one-parameter family of Riemannian metrics and 2-forms with  
\begin{align*}
    &\frac{\partial}{\partial t}\Big|_{t=0}g_t=h,\quad \frac{\partial}{\partial t}\Big|_{t=0}b_t=K,  \quad  (g_0,b_0)=(g,b).
\end{align*}
Denote $\gamma=h-K$. In \cite{KK} Theorem 1,1, the author defined the operator
\begin{align*}
    \overline{N_f}(\gamma)=\frac{1}{2}\overline{\Delta}_f\gamma+\mathring{R}^+(\gamma)+\frac{1}{2}\overline{\divg}_f^*\overline{\divg}_f\gamma+(\nabla^+)^2\phi,
\end{align*}
where $\phi$ is the unique solution of 
\begin{align*}
    \triangle_f \phi=\overline{\divg}_f\overline{\divg}_f\gamma,\quad \int_M \phi e^{-f}dV_g=0.
\end{align*}
Here, the definition of $\overline{\divg}_f$, $\overline{\divg}_f^*$, $\overline{\triangle}_f$ are given in Definition \ref{DD} and $\langle \mathring{R}^+(\gamma),\gamma \rangle=R^+_{iklj}\gamma_{ij}\gamma_{kl}$, $R^+$ is the Bismut curvature given in \Cref{P3}. Then, the second variation of $\lambda$ at $(g,b)$ is given by 
\begin{align*}
     \nonumber\frac{d^2}{dt^2}\Big|_{t=0}\lambda=\int_M \Big\langle \gamma,  \overline{N_f}(\gamma) \Big\rangle e^{-f}dV_g.
\end{align*} 

Naturally, we decompose the space of variations via
\begin{align*}
    \otimes^2T^*M=\ker\overline{\divg}_f\oplus \text{Im}\overline{\divg}_f^*.
\end{align*}
In this work, we prove the following. 
 \begin{theorem}\label{MT}
Suppose $\mathcal{G}(g,b)$ is a steady gradient generalized Ricci soliton on a smooth compact manifold $M$. For any 2-tensor $\gamma=\overline{\divg}^*_f(u,v)$ with $(u,v)\in T^*M\times T^*M$, we have 
\begin{align*}
    \overline{N}_f(\gamma)=0.
\end{align*}
Hence, the second variation formula of $\lambda$ at $(g,b)$ reduces to 
\begin{align*}
     \nonumber\frac{d^2}{dt^2}\Big|_{t=0}\lambda=\int_M \Big\langle \gamma, \frac{1}{2}\overline{\Delta}_f\gamma+\mathring{R}^+(\gamma) \Big\rangle e^{-f}dV_g,
\end{align*} 
where $\gamma\in\ker\overline{\divg}_f$.
\end{theorem}

Our second main result concerns the stability of generalized Ricci solitons. We introduce the following concepts.
\begin{enumerate}
    \item \textbf{Linearly stable: } We say that $(g,b)$ is \emph{linearly stable} if 
\begin{align*}
    \frac{d^2}{dt^2}\lambda(\gamma)\leq 0 \quad \text{ for all variation $\gamma$ at $(g,b)$}.
\end{align*}

    \item \textbf{Dynamically stable:} We say that a steady gradient generalized Ricci soliton $(g,b)$ is \emph{dynamically stable} if for any neighborhood $\mathcal{U}$ of $(g,b)$, there exists a smaller neighborhood $\mathcal{V}\subset\mathcal{U}$ such that for any solution of generalized Ricci flow $(g_t,b_t)$ starting in $\mathcal{V}$, there exists some automorphism $(\varphi_t,B_t)$ (c.f Definition 2.2 and (\ref{Aut})) whose modified flow $(\varphi_t,B_t)\cdot (g_t,b_t)$ stays in $\mathcal{U}$ for all time and converges to a critical point $(g_\infty,b_\infty)$ of $\lambda$ with $\lambda(g,b)=\lambda(g_\infty,b_\infty)$. 
    
    \item \textbf{Local maximum:} There exists a neighborhood $\mathcal{U}$ of $(g,b)$ such that $\lambda(g,b)$ is a maximum in the neighborhood $\mathcal{U}$. 
\end{enumerate}
Since $\lambda$ increases monotonically along the generalized Ricci flow, it is immediate that:
\begin{align*}
    \text{Dynamically stable}\Longrightarrow\text{Local maximum}\Longrightarrow \text{Linearly stable}
\end{align*}

However, these implications are not equivalences in general. In \cite{K}, it was shown that the local maximum property implies dynamical stability. In this work, we prove the converse implication i.e., that linear stability implies local maximum under the assumption that all infinitesimal generalized solitonic deformations are integrable (c.f Definition \ref{IGSD}). The precise statement is as follows.
\begin{theorem}\label{T1}
   Let $M$ be a compact manifold and $E\cong (TM\oplus T^*M)_{H_0}$ be an exact Courant algebroid with a background closed 3-form $H_0$. Suppose that $\mathcal{G}_0$ is a steady gradient generalized Ricci soliton and all infinitesimal generalized solitonic deformations are integrable. If $\mathcal{G}_0$ is linearly stable, then there exists a $C^{2,\alpha}-$neighborhood $\mathcal{U}$ of $\mathcal{G}_0$ such that for all $\mathcal{G}\in\mathcal{U}$,
   \begin{align*}
       \lambda(\mathcal{G})\leq \lambda(\mathcal{G}_0).
   \end{align*}
   The equality holds if and only if $\mathcal{G}$ is isomorphic to some steady gradient generalized Ricci soliton. 
\end{theorem}

A key ingredient in the proof of \Cref{T1} is the generalized slice theorem. In Riemannian geometry, Ebin’s work \cite{MR0267604} constructed a local slice for the space of Riemannian metrics, which has been fundamental in the study of moduli spaces of Einstein metrics. Recently, Rubio and Tipler extended this construction to generalized geometry in \cite{Rubio_2019}, proposing a generalized Ebin slice theorem for exact Courant algebroids. Their result builds on Ebin’s original framework and adapts it to the context of generalized metrics.

In our setting, the author proved the generalized slice theorem based on the $f$-twisted inner product (\ref{6}). Moreover, in this work the author further show that it suffices to consider the slice to be affine. Motivated by the work of Ache \cite{Ache2012OnTU} and Gursky, Viaclovsky \cite{Gursky2011RigidityAS}, the author extends the generalized slice theorem and shows that
\begin{theorem}
    Let $\mathcal{G}$ be a generalized metric on an exact Courant algebroid $E$ and let $f$ be $\Isom(\mathcal{G})$ invariant. There exists a neighborhood $\mathcal{U}$ of $\mathcal{G}$ such that for any $\widetilde{\mathcal{G}}$, we can find $(X,\omega)\in \mathfrak{gdiff}_H^e$ so that
    \begin{align*}
        \overline{\divg}_f\big( (\varphi_X,B)\cdot\widetilde{ \mathcal{G}}- \mathcal{G} \big)=0,
    \end{align*}
    where $(\varphi_X,B)\in \GDiff_H^e$ is the flow generated by $(X,\omega)$ at time 1. Here, $\cdot$ denotes the $\GDiff_{H}$ action and the definition of $\GDiff_H^e$ is given in (\ref{Aut}) and (\ref{GA}).
\end{theorem}

Under the assumption of integrability, we further prove that the generalized Ricci flow (GRF) exhibits exponential convergence near a steady soliton.
\begin{theorem} 
    
    Let $M$ be a compact manifold and $E\cong (TM\oplus T^*M)_{H_c}$ be an exact Courant algebroid with a background closed 3-form $H_c$. Suppose that $\mathcal{G}_c$ is a steady gradient generalized Ricci soliton and all infinitesimal generalized solitonic deformations are integrable. If $\mathcal{G}_c$ is linearly stable and $k\geq 2$, then for every $C^k$-neighborhood $\mathcal{U}$ of $\mathcal{G}_c$ , there exists some $C^{k+2}$-neighborhood $\mathcal{V}$ such that the following holds:

     For any GRF $\mathcal{G}_t$ starting at $\mathcal{G}_0\in \mathcal{V}$, there exists a family of automorphism $\{(\varphi_t,B_t)\}\in \GDiff_{H_c}$ such that the modified flow $(\varphi_t,B_t)\cdot\mathcal{G}_t$ stays in $\mathcal{U}$ for all time and converges exponentially to a steady gradient generalized Ricci soliton$\mathcal{G}_\infty$. More precisely, there exists constants $C_1,C_2$ such that for all $t\geq 0$, 
     \begin{align*}
         \|(\varphi_t,B_t)\cdot\mathcal{G}_t-\mathcal{G}_\infty\|_{C^k_{\mathcal{G}_c}}\leq C_1 e^{-C_2t}.
     \end{align*}

\end{theorem}

In the next part, we focus on the pluriclosed setting. Let $(M,g,H,J,f)$ be a pluriclosed steady soliton with a closed 3-form $H=-d^c\omega$ (c.f  Definition \ref{DDD}). Consider a family $(M,g_t,H_t,J_t,f_t)$ with 
\begin{align*}
    H_t=H_0+db_t,\quad (M,g_0,H_0,J_0,b_0)=(M,g,H,J,0),
\end{align*}
and $f_t$ is the minimizer of $\lambda(g_t,b_t)$. Denote 
\begin{align*}
   \frac{\partial }{\partial t}\big|_{t=0} g=h, \quad  \frac{\partial }{\partial t}\big|_{t=0} b=\beta,\quad   \frac{\partial }{\partial t}\big|_{t=0}\omega =\phi, \quad  \frac{\partial }{\partial t}\big|_{t=0} J =I,
\end{align*}
where $I$ is an infinitesimal variation of complex structure. In \cite{KKK}, we examined such variations in detail and showed that for any $\gamma\in\ker\overline{\divg}_f$, there exists a 2-form $\xi$ with $(d^B_f)^*\xi=0$ and a symmetric 2-tensor $\eta$ with $\divg_f^B\eta=0$, satisfying (\ref{etacommute}) such that
\begin{align*}
    \gamma(X,Y)=\xi(X,JY)+\eta(X,Y). 
\end{align*}
where $\divg_f^B$ is the divergence operator and $(d^B_f)^*$ is the formal adjoint of $d^B$ with respect to the Bismut connection (\ref{20}). In fact, $\xi=\phi+\beta\circ J$ and $\eta=\omega\circ I$. Furthermore, the second variation formula in the pluriclosed setting is given by
\begin{align*}
     \frac{d^2}{dt^2}\Big|_{t=0}\lambda(\gamma)&= -2\|d^*_f\xi\|_f^2-\frac{1}{6}\|d\xi-C\|_f^2 +\int_M |\eta|^2\big( -\frac{1}{2}S_B+\mathcal{L}_Vf \big)e^{-f}dV_g, 
\end{align*} 
where $\gamma=\xi\circ J+\eta$, $C$ is defined in (\ref{C}), $\|\cdot\|_f$ denotes the norm of $f$-twisted $L^2$ inner product (\ref{6}),  $S_B=\tr_\omega \rho_B$ denotes the Bismut scalar curvature and the vector field $V=\frac{1}{2}(\theta^\sharp-\nabla f)$. In this work, we introduce an elliptic operator 
\begin{align*}
    \overline{L}_f(\gamma)=\frac{1}{2}\overline{\Delta}_f\gamma+\mathring{R}^+(\gamma).
\end{align*}
By \Cref{MT}, we know that the operator $\overline{L}_f(\gamma)$ is closely related to the second variation. In the pluriclosed setting, we prove that 
\begin{proposition}
   Let $(M,g,H,J,f)$ be a compact steady pluriclosed soliton. For any two form $\xi$,
   \begin{align*}
        \overline{L}_f(\xi)=0\Longleftrightarrow d\xi=0,\quad d^*_f\xi=0,\quad (d^B_f)^*\xi=0.
   \end{align*}
   Therefore, if $\gamma=\xi\circ J\in IGSD$, then $\gamma\circ J$ is $f-$twisted harmonic, i.e, $\Delta_{d,f}(\gamma\circ J)=0$.
\end{proposition}
This proposition establishes a strong link between the operator $ \overline{L}_f$ and the harmonicity conditions associated with the pluriclosed flow. In the special case where the manifold is Bismut-flat, we obtain even stronger structural results.
\begin{corollary}\label{CC1}
    Let $(M,g,H,J)$ be a compact Bismut-flat manifold. For any two form $\xi$,
    \begin{align*}
         \overline{L}_f(\xi)=0\Longleftrightarrow \xi \text{ is $\nabla^{\pm}$-parallel.}
    \end{align*}
  Therefore, if $\gamma=\xi \circ J\in IGSD$ then $\gamma$ is $\nabla$-parallel and also $\nabla^\pm$-parallel.
\end{corollary}

The \Cref{CC1} leads to the following dynamical stability result for the pluriclosed flow.
\begin{theorem}
Suppose $(M,g,H,J)$ is a compact, pluriclosed Bismut-flat manifold. Then, $(M,g,H,J)$ is pluriclosed flow dynamically stable and the convergence rate is exponential.     
\end{theorem}

To establish the local maximum of the entropy functional $\lambda$ in the pluriclosed soliton setting, we consider variations restricted to a fixed real $(1,1)-$Aeppli cohomology class (c.f Definition \ref{Aeppli}) and observe that steady pluriclosed soliton reaches the local maximum. In conclusion, we have the dynamical stable property of steady pluriclosed soliton.
\begin{theorem}
      Let $(M,g_E,b_E,J,f_E)$ be a compact, steady pluriclosed soliton, $H_E=-d^c\omega_E$ and the first Chern class $c_1=0\in H^{1,1}_A$. Then, there exists a neighborhood $\mathcal{U}$ of $(g_E,b_E)$ such that if the pluriclosed flow starts at $(g_0,b_0)\in\mathcal{U}$ with $[\omega_E]=[\omega_0]\in H^{1,1}_A$, there exists a family of automorphism $\{(\varphi_t,B_t)\}\in \GDiff_{H_E}$ such that the modified flow $(\varphi_t,B_t)\cdot(g_t,b_t)$ stays in $\mathcal{U}$ for all time and converges exponentially to some steady pluriclosed soliton $(\varphi_\infty,B_\infty)\cdot (g_E,b_E)$.
\end{theorem}

The layout of this paper is as follows: In section 2, we review background material, including the generalized slice theorem, properties of Bismut connections and curvature, and the formulation of the generalized Ricci flow. In section 3, we focus on the second variation operator. In section 4, We explore the local maximum property of $\lambda$, and prove \Cref{T1} relating linear and dynamical stability. In section 5, we focus on the pluriclosed setting and study the infinitesimal generalized solitonic deformation. In section 6, we discuss the dynamical stability of the pluriclosed flow.

\textbf{Acknowledgements:} This work is written when the author is a math Ph.D. student at the University of California-Irvine and it is part of his Ph.D. thesis. I am grateful to my advisor Jeffrey D.~Streets for his helpful advice. His suggestions play an important role in this work. 

\section{Preliminary}

\subsection{Notation}
In this work, the convention of Riemann curvature is given by 
\begin{align*}
    R(X,Y,Z,W)=\langle \nabla_X\nabla_YZ-\nabla_Y\nabla_XZ-\nabla_{[X,Y]}Z,W \rangle, \quad R_{ijkl}=R(e_i,e_j,e_k,e_l).
\end{align*}
We adopt the following notations.

\begin{itemize}
    \item For any 2-tensor $\gamma$,
\begin{align*}
    \mathring{R} (\gamma)_{ij}= R_{iklj}\gamma_{kl}. 
\end{align*}
    \item The $f$-twisted $L^2$ inner product is given by
\begin{align}
    \Big( \gamma_1,\gamma_2\Big)_{f}\coloneqq\int_M \langle \gamma_1,\gamma_2 \rangle_g  e^{-f} dV_g, \label{6}
\end{align}
where $\langle\kern0.3em,\kern0.3em\rangle_g$ denotes the standard inner product induced by a Riemannian metric $g$ and $\gamma_1,\gamma_2$ are tensors of same type. In particular, we mainly focus on the case when $\gamma$ is a 2-tensor.
\item The $f$-twisted Laplacian $\Delta_f$ is given by
\begin{align*}
    \Delta_f=\Delta-\nabla f\cdot \nabla. 
\end{align*}
 \item Let $S^pM$ denote the bundle of symmetric $(0,p)-$tensor. The $f$-twisted divergence operator $\divg_f: C^\infty(S^pM)\rightarrow  C^\infty(S^{p-1}M)$ is given by 
  \begin{align*}
     (\divg_f T)(X_1,...,X_{p-1})=-\sum_{i=1}^n (\nabla_{e_i}T)(e_i,X_1,...,X_{p-1})+T(\nabla f,X_1,...,X_{p-1}).
 \end{align*}
 \item On a complex manifold $(M,J)$, we denote
\begin{align*}
    (\gamma\circ J)(\cdot,\cdot)=\gamma(\cdot,J\cdot),
\end{align*}
for any 2-tensor $\gamma$. In addition, in the following we will use the index $i,j,k,...$ to represent real coordinates and the index $\alpha,\beta,\gamma,...$ to represent complex coordinates. 
\end{itemize}

\subsection{Generalized Geometry}
In this section, we review some basic definitions and properties of generalized geometry. More details can be found in \cite{GRF}.

\begin{defn}\label{D2}
A \emph{Courant algebroid} is a vector bundle $E\longrightarrow M$ with a nondegenerate bilinear form $\langle \cdot,\cdot\rangle$, a bracket $[\cdot,\cdot]$ on $\Gamma(E)$ and a bundle map $\pi:E\longrightarrow TM$ satisfies that for all $a,b,c\in\Gamma(E)$, $f\in C^\infty(M)$,
\begin{itemize}
    \item $[a,[b,c]]=[[a,b],c]+[b,[a,c]]$.
    \item $ \pi[a,b]=[\pi(a),\pi(b)].$
    \item $[a,fb]=f[a,b]+\pi(a)fb.$
    \item $\pi(a)\langle b,c\rangle=\langle [a,b],c \rangle+\langle a,[b,c] \rangle.$
    \item $[a,b]+[b,a]=\mathcal{D}\langle a,b \rangle $ where $\mathcal{D}: C^\infty(M)\to \Gamma(E)$ is given by $\mathcal{D}(\phi)\coloneqq\pi^*(d\phi)$.
\end{itemize}
We say a Courant algebroid $E$ is \emph{exact} if we have the following exact sequence of vector bundles 
\[  \begin{tikzcd}
  0 \arrow[r] & T^*M  \arrow[r, "\pi^*"] & E \arrow[r, "\pi"] & TM  \arrow[r] & 0 .
\end{tikzcd}
\]
\end{defn}

\begin{defn}\label{D3}
Let $E$ be an exact Courant algebroid. The \emph{automorphism group} $\Aut(E)$ of $E$ is a pair $(f,F)$ where $f\in\Diff(M)$ and $F\colon E\to E$ is a bundle map such that for all $u,v\in \Gamma(E)$
\begin{itemize}
    \item $ \langle Fu,Fv \rangle=f_{*}\langle u,v\rangle.$
    \item $[Fu,Fv]=F[u,v].$
    \item $\pi_{TM}\circ F=f_{*}\circ \pi_{TM}.$
\end{itemize}

\end{defn}

\begin{defn}\label{D4}
Given a smooth manifold $M$ and an exact Courant algebroid $E$ over $M$, a \emph{generalized metric} on $E$ is a bundle endomorphism $\mathcal{G}\in \Gamma(\End(E))$ satisfying 

\begin{itemize}
    \item $\langle \mathcal{G}a,\mathcal{G}b \rangle=\langle a,b\rangle.$
    \item $ \langle \mathcal{G}a,b \rangle=\langle a,\mathcal{G}b\rangle.$
    \item $ \langle \mathcal{G}a,b \rangle \text{  is symmetric and positive definite for any $a,b\in E$.}$
\end{itemize}

\end{defn}

\begin{example}
The most common and important example of Courant algebroids is $TM\oplus T^*M$. In this case, we define a nondegenerate bilinear form $\langle\cdot,\cdot\rangle$ and a bracket $[\cdot,\cdot]$ on $TM\oplus T^*M$ by
\begin{align*}
    \langle X+\xi,Y+\eta \rangle&\coloneqq\frac{1}{2}(\xi(Y)+\eta(X)),
    \\ [X+\xi,Y+\eta]_H&\coloneqq[X,Y]+L_X\eta-i_Y d\xi+i_Yi_XH
\end{align*}
where $X,Y\in TM$,  $\xi,\eta\in T^*M$ and $H$ is a 3-form. Define $\pi$ to be the standard projection, one can check that $(TM\oplus T^*M)_H\coloneqq(TM\oplus T^*M,\langle\cdot,\cdot\rangle,[\cdot,\cdot]_H,\pi)$ satisfies the Courant algebroid conditions. Moreover, its automorphism groups are given as follows.
\begin{align*}
   \GDiff_H=\{(f,\overline{f}\circ e^B): f\in\Diff(M), B\in \Omega^2 \text{  such that  } f^*H=H-dB\},
\end{align*}
where
\begin{align*}
    \overline{f}&=\begin{pmatrix} f_{\star} & 0 \\ 0 & (f^*)^{-1} \end{pmatrix}: X+\alpha\longmapsto f_{*}X+(f^*)^{-1}(\alpha),
    \\ e^B&=\begin{pmatrix} Id & 0 \\ B & Id \end{pmatrix}: X+\alpha\longmapsto X+\alpha+i_XB  \qquad \text{  for any $X\in TM$ and $\alpha\in T^*M$}.
\end{align*}
The product of automorphisms is given by
\begin{align*}
    (f,F)\circ (f',F')=\overline{f\circ f'}\circ e^{B'+f'^*B} \quad \text{  where } F=\overline{f}\circ e^B,\quad F'=\overline{f'}\circ e^{B'}.
\end{align*}
In the following, we will denote $\GDiff_H$ to be the automorphism group of $(TM\oplus T^*M)_H$, $\mathcal{GM}$ to be the space of all generalized metrics, and $\mathcal{M}$ to be the space of all Riemannian metrics. 
\end{example}

Recall that in \cite{GRF} Proposition 2.10, we see that for any exact Courant
Courant algebroid $E$ with a isotropic splitting $\sigma$, $E\cong_\sigma (TM\oplus T^*M)_H$ where 
\begin{align*}
    H(X,Y,Z)=2\langle [\sigma X,\sigma Y], \sigma Z \rangle \quad X,Y,Z \in TM. 
\end{align*}
Therefore, we see that
\begin{align}
    \Aut(E)\cong_{\sigma}\GDiff_H=\{(\varphi,B)\in\Diff(M)\ltimes\Omega^2: \varphi^*H=H-dB\}. \label{Aut}
\end{align}
Moreover, we have the following proposition.

\begin{proposition}[\cite{GRF} Proposition 2.38 and 2.40] \label{P2}
Let $E$ be an exact Courant algebroid. The space of all generalized metrics $\mathcal{GM}$ on $E$ is isomorphic to $\mathcal{M}\times \Omega^2$.
\end{proposition}
\begin{remark}\label{R1}
Fix a background 3-form $H_0$ such that $E\cong (TM\oplus T^*M)_{H_0}$, the proof of \Cref{P2} implies that the 3-form $H$ of any generalized metric $\mathcal{G}=\mathcal{G}(g,b)$ is induced by an isotropic splitting $\sigma(X)=X+i_X b$ and then we have $H=H_0+db$. (See \cite{K} Remark 2.7 for more details.)   
\end{remark}

In the generalized geometry, we define the $\GDiff_H$ action on generalized metrics by
\begin{align}
    \rho_{\mathcal{GM}}:\nonumber\quad &\GDiff_H \times \mathcal{GM} \longrightarrow \mathcal{GM}
    \\&((\varphi,B),(g,b))\longmapsto (\varphi^*g,\varphi^*b-B). \label{GA}
\end{align}
Here, we note that $\mathcal{GM}\cong \mathcal{M}\times \Omega^2 $ so in the following, we will always denote a generalized metric $\mathcal{G}$ by $\mathcal{G}(g,b)$ for some $(g,b)\in \mathcal{M}\times \Omega^2$. Therefore,
\begin{align} \label{TGM}
    T_{\mathcal{G}}\mathcal{GM}= T_{\mathcal{G}}(\mathcal{M}\times \Omega^2)=\Gamma(S^2M)\times \Omega^2=\otimes^2 T^*M.
\end{align}
\begin{remark}\label{Gamma}
 In the following, we usually denote $\gamma\in T_{\mathcal{G}}\mathcal{GM}$. Due to (\ref{TGM}), we see that $\gamma\in \otimes^2T^*M$ and we write $\gamma=h-K$, where $h$ is the symmetric part and $-K$ is skew-symmetric part.    
\end{remark}

\begin{defn}\label{D8}
The group of \emph{generalized isometries} (exact generalized isometries) of $\mathcal{G}\in\mathcal{GM}$ is the isotropy group of $\mathcal{G}$ under the $\GDiff_H$ ($\GDiff^e_H$)-action and is denoted by $\Isom_H(\mathcal{G})$ ($\Isom^e_H(\mathcal{G})$).  
\end{defn}

In \cite{Rubio_2019}, Rubio and Tipler proposed the generalized Ebin's slice theorem based on the $L^2$ inner product. In \cite{K}, we proved the generalized slice theorem based on the $f$-twisted inner product (\ref{6}). The precise statement is as follows.
\begin{theorem}[\cite{K} Theorem 2.14]\label{slice}
Let $\mathcal{G}$ be a generalized metric on an exact Courant algebroid $E$ and $f$ be $\Isom(\mathcal{G})$ invariant, then there exists a submanifold $S^f_\mathcal{G}$ of $\mathcal{GM}$ such that 
\begin{itemize}
    \item $\forall F\in \Isom_H(\mathcal{G}),\, F\cdot S^f_\mathcal{G}=S^f_\mathcal{G}$.
    \item $ \forall F\in \GDiff_H, \text{if $(F\cdot S^f_\mathcal{G})\cap S^f_\mathcal{G}=\emptyset $ then }F\in \Isom_H(\mathcal{G})$.
    \item There exists a local cross section $\chi$ of the map $F\longmapsto \rho_{\mathcal{GM}}(F,\mathcal{G})$ on a neighborhood $\mathcal{U}$ of $\mathcal{G}$ in the orbit space $\mathcal{O}_{\mathcal{G}}=\GDiff_H\cdot\mathcal{G}$ such that the map from $\mathcal{U}\times S^f_{\mathcal{G}}\longrightarrow \mathcal{GM}$ given by $(V_1,V_2)\longmapsto \rho_{\mathcal{GM}}(\chi(V_1),V_2)$ is a homeomorphism onto its image.
\end{itemize}
where $\Isom_H(\mathcal{G})$ is the isotropy group of $\mathcal{G}$ under the $\GDiff_H$-action and is called the group of generalized isometries of $\mathcal{G}\in\mathcal{GM}$. Moreover, the tangent space of the generalized slice on a generalized metric $\mathcal{G}(g,b)$ is given by
\begin{align}
    T_{\mathcal{G}}S^f_\mathcal{G}=\{\gamma=h-K\in\otimes^2T^*M: \kern0.5em (\divg_fh)_l=\frac{1}{2}K_{ab}H_{lab},  \kern0.5em d_f^*K=0 \}. \label{10}
\end{align}
\end{theorem}

\subsection{Bismut connections and curvatures}

In this subsection, let us briefly review the definition of Bismut connections. More details could be found in \cite{GRF}, \cite{KK}.
\begin{defn}\label{DD}
Let $(M,g,H)$ be a Riemannian manifold and $H$ is a closed 3-form. 
\begin{itemize}
    \item The \emph{Bismut connections} $\nabla^{\pm}$ associated to $(g,H)$ are defined as 
\begin{align}\label{PM}
    \langle \nabla_X^{\pm}Y,Z \rangle=\langle \nabla_XY,Z \rangle\pm \frac{1}{2}H(X,Y,Z). 
\end{align}
Here, $\nabla$ is the Levi-Civita connection associated with $g$, i.e., $ \nabla^{\pm}$ are the unique compatible connections with torsion $\pm H$.
\item The \emph{mixed Bismut connection} is a connection $\overline{\nabla}$ on 2-tensors is defined as
\begin{align}
    (\overline{\nabla}_X\gamma)(Y,Z)=\nabla_X\Big(\gamma(Y,Z)\Big)-\gamma(\nabla^-_XY,Z)-\gamma(Y,\nabla^+_XZ). \label{MBC}
\end{align}
\item The \emph{$f$-twisted divergence operator} $\overline{\divg}_f:  \otimes^2T^*M\longrightarrow T^*M\times T^*M$ on 2-tensors with respect to $\overline{\nabla}$ is given by
\begin{align}
       \overline{\divg}_f\gamma&=(u,v), \label{bar1}
\end{align}
where $u_l=(\nabla^+)^m\gamma_{ml}-\nabla_mf \gamma_{ml}$ and $v_l=(\nabla^-)^m\gamma_{lm}-\nabla_mf \gamma_{lm}$
\item The \emph{$f$-twisted divergence operator} $\overline{\divg}_f:  T^*M\times T^*M\longrightarrow C^\infty(M)$ on $T^*M\times T^*M$ with respect to $\overline{\nabla}$ is given by
\begin{align}
       \overline{\divg}_f(u,v)=\frac{1}{2}(\divg_f u+\divg_f v)\label{bar2} 
\end{align}
and its formal adjoint $\overline{\divg}_f^*$ is 
\begin{align}
    \overline{\divg}_f^*(u,v)_{ij}=-(\nabla^+)^iu_j-(\nabla^-)^jv_i. \label{dbars}
\end{align}
\item The \emph{Laplace operator of the mixed Bismut connection} $\overline{\Delta}_f$ is defined by 
\begin{align}
    \overline{\Delta}_f=-\overline{\nabla}^{*_f}\overline{\nabla}\label{bar3} 
\end{align}
where $\overline{\nabla}^{*_f}$ is the formal adjoint of $\overline{\nabla}$ with respect to $f$-twisted $L^2$ inner product (\ref{6}). More precisely, suppose $\gamma$ is a 2-tensor,
\begin{align}
\overline{\Delta}_f\gamma_{ij}=\Delta_f\gamma_{ij}-H_{mjk}\nabla_m\gamma_{ik}+H_{mik}\nabla_m\gamma_{kj}-\frac{1}{4}(H_{jl}^2\gamma_{il}+H_{il}^2\gamma_{lj})-\frac{1}{2}H_{mkj}H_{mli}\gamma_{lk}. \label{Lbar}  
\end{align}
\end{itemize}

\end{defn}

Following the definitions, we are able to compute the curvature tensor of the Bismut connection and the Bianchi identity of Bismut curvature.
\begin{proposition}[\cite{GRF} Proposition 3.18]\label{P3}
Let $(M^n,g,H)$ be a Riemannian manifold with $H\in \Omega^3$ and $dH=0$, then for any vector fields $X,Y,Z,W$ we have
\begin{align*}
    \nonumber Rm^+(X,Y,Z,W)=& \kern0.2em Rm(X,Y,Z,W)+\frac{1}{2}\nabla_XH(Y,Z,W)-\frac{1}{2}\nabla_YH(X,Z,W)
    \nonumber\\&-\frac{1}{4}\langle H(X,W),H(Y,Z)\rangle+\frac{1}{4}\langle H(Y,W),H(X,Z)\rangle, 
    \\\Rc^+&=\Rc-\frac{1}{4}H^2-\frac{1}{2}d^*H, \quad R^+=R-\frac{1}{4}|H|^2, 
\end{align*}
where $H^2(X,Y)=\langle i_XH,i_YH\rangle$. Here $Rm^+$, $\Rc^+$, $R^+$ denote the Riemannian curvature, Ricci curvature, and scalar curvature with respect to the Bismut connection $\nabla^+$. In particular, if $(M,g,H)$ is Bismut-flat, then 
\begin{align}
    Rm(X,Y,Z,W)=\frac{1}{4}\langle H(X,W),H(Y,Z)\rangle-\frac{1}{4}\langle H(Y,W),H(X,Z)\rangle\quad \text{and } \quad \nabla H=0 \label{R}
\end{align}
for any vector fields $X,Y,Z,W$.
\end{proposition}
\begin{proposition}[\cite{GRF} Proposition 3.20]
Let $(M^n,g,H)$ be a Riemannian manifold with $H\in \Omega^3$ and $dH=0$, then for any vector fields $X,Y,Z,W$ we have
\begin{align}
    \nonumber\sum_{\sigma(X,Y,Z)}R^+(X,Y,Z,W)&=(\nabla^+_WH)(X,Y,Z)- \sum_{\sigma(X,Y,Z)}g\big(H(X,Y),H(Z,W)\big)
    \\&=\frac{1}{2}\Big( \sum_{\sigma(X,Y,Z)}(\nabla^+_XH)(Y,Z,W)+(\nabla^+_WH)(X,Y,Z) \Big).\label{BianchiBismut}
\end{align}
\end{proposition}

\subsection{Generalized Einstein--Hilbert functional and Generalized Ricci flow}
In this subsection, we discuss the generalized Einstein--Hilbert functional and generalized Ricci flow. Most of the contents can be found in \cite{GRF} Chapter 4 and 6.

\begin{defn}\label{G}
Let $E$ be an exact Courant algebroid over a smooth manifold $M$ and $H_0$ is a background closed 3-form. A one-parameter family of generalized metrics $\mathcal{G}_t=\mathcal{G}_t(g_t,b_t)$
is called a \emph{generalized Ricci flow} if 
\begin{align*}
   &\frac{\partial}{\partial t}g=-2\Rc+\frac{1}{2}H^2,\nonumber
   \\&\frac{\partial}{\partial t}b=-d^*H \quad \text{ where $H=H_0+db$.} 
\end{align*}
Equivalently, the generalized Ricci flow can also be expressed as
\begin{align*}
   &\frac{\partial}{\partial t}(g-b)=-2\Rc^+.
\end{align*}
where $\Rc^+$ denotes the Ricci curvature with respect to the Bismut connection $\nabla^+$.
\end{defn}

\begin{defn}\label{D10}
Given a smooth compact manifold $M$ and a background closed 3-form $H_0$, the \emph{generalized Einstein--Hilbert functional} $\mathcal{F}: \Gamma(S^2M)\times \Omega^2 \times C^\infty (M)\to \mathbb{R}$ is given by 
\begin{align*}
    \mathcal{F}(g,b,f)=\int_M (R-\frac{1}{12}|H_0+db|^2+|\nabla f|^2)e^{-f}dV_g.
\end{align*}
Also, we define 
\begin{align*}
  \lambda(g,b)=\inf\Big\{\mathcal{F}(g,b,f)\big|\, f\in C^\infty(M),\,\int_M e^{-f}dV_g=1\Big \}.
\end{align*}

\end{defn}

\begin{remark}
 For any $(g,b)$, the minimizer $f$ is always achieved. Moreover, $\lambda$ satisfies that
\begin{align}
    \lambda(g,b)=R-\frac{1}{12}|H_0+db|^2+2\Delta f-|\nabla f|^2 \label{lam}
\end{align}
and it is the lowest eigenvalue of the Schrödinger operator $-4\Delta+R-\frac{1}{12}|H_0+db|^2.$   
\end{remark}

Let $(g_t, b_t)$ be a smooth family of pairs of Riemannian metrics and 2-forms on a smooth compact manifold $M$. Suppose
\begin{align*}
   &\frac{\partial}{\partial t}\Big|_{t=0}g_t=h,\quad \frac{\partial}{\partial t}\Big|_{t=0}b_t=K,  \quad  (g_0,b_0)=(g,b).
\end{align*}
In the following, we denote $\gamma=h-K\in\otimes^2 T^*M$ and $\gamma$ is called a general variation. The first variation formula of the generalized Einstein--Hilbert functional (c.f \cite{K} Theorem 3.3) is given by
\begin{align} \label{fv}
    \nonumber\frac{d}{dt}\Big|_{t=0}\lambda(g_t,b_t)&=\int_M \Big[\langle -\Rc+\frac{1}{4}H^2-\nabla^2f,h \rangle-\frac{1}{2}\langle d^*H+i_{\nabla f}H,K\rangle \Big]e^{-f}dV_g
    \\&=\int_M -\langle \gamma, \Rc^{H,f} \rangle e^{-f}dV_g, 
\end{align}
where $\Rc^{H,f}=\Rc-\frac{1}{4}H^2+\nabla^2f-\frac{1}{2}( d^*H+i_{\nabla f}H)$ is the twisted Bakry-Emery curvature. Motivated by the Ricci flow, we define the stationary points to be the steady generalized Ricci solitons. (For more details about motivations, readers can consult with \cite{GRF} and \cite{K}.)

\begin{defn}\label{soliton}
Let $M$ be a compact manifold and $E\cong (TM\oplus T^*M)_{H_0}$ be an exact Courant algebroid with a background closed 3-form $H_0$. Suppose that $\mathcal{G}(g,b)$ is a generalized metric.
\begin{itemize}
    \item $\mathcal{G}(g,b)$ is called a \emph{steady gradient generalized Ricci soliton} if 
\begin{align}
   0=\Rc-\frac{1}{4}H^2+\nabla^2 f, \quad 0=d_g^*H+i_{\nabla f}H  \label{s},
\end{align}
where $H=H_0+db$ and $f$ is the minimizer of $\lambda(g,b)$.
\item $\mathcal{G}(g,b)$ is called a \emph{generalized Einstein metric} if $\mathcal{G}(g,b)$ is a steady gradient generalized Ricci soliton with the constant minimizer $f$.

\end{itemize}

\end{defn}

\section{Second Variation Operator}

In this section, we aim at discussing the second variation of the generalized Einstein--Hilbert functional. First of all, let us recall the second variation and infinitesimal deformation. 

\subsection{Second variation and Infinitesimal deformations}

In \cite{KK} Theorem 1.1, the author derived the second variation formula of the generalized Einstein--Hilbert which is given in the following
\begin{theorem}[\cite{KK} Theorem 1.1]
 Let $M$ be a compact manifold, $E\cong (TM\oplus T^*M)_{H_0}$ be an exact Courant algebroid with a background closed 3-form $H_0$ and $\mathcal{G}(g,b)$ is a compact steady gradient generalized Ricci soliton. Suppose $\mathcal{G}(g_t,b_t)$ is a one-parameter family of generalized metrics such that 
\begin{align*}
    \frac{\partial}{\partial t}\Big|_{t=0}g_t =h, \quad \frac{\partial}{\partial t}\Big|_{t=0}b_t=K,\quad  (g_0,b_0)=(g,b).
\end{align*}
Denote $\gamma=h-K$. Let $\phi$ be the unique solution of 
\begin{align*}
    \Delta_f \phi =\overline{\divg}_f\overline{\divg}_f\gamma,\quad \int_M \phi e^{-f}dV_g=0,
\end{align*}
where the definition of $\overline{\divg}_f$ is given in (\ref{bar1}) and (\ref{bar2}). We define
\begin{align}\label{N}
    \overline{N_f}(\gamma)=\frac{1}{2}\overline{\Delta}_f\gamma+\mathring{R}^+(\gamma)+\frac{1}{2}\overline{\divg}_f^*\overline{\divg}_f\gamma+(\nabla^+)^2\phi.
\end{align}
where $\overline{\divg}_f^*$ is the formal adjoint of $\overline{\divg}_f$ with respect to (\ref{6}), $\overline{\Delta}_f$ 
 is defined in (\ref{bar3}), $\langle \mathring{R}^+(\gamma),\gamma \rangle=R^+_{iklj}\gamma_{ij}\gamma_{kl}$ and $R^+$ is the Bismut curvature given in \Cref{P3}. Then, the second variation of $\lambda$ at $(g,b)$ is given by 
\begin{align*}
     \nonumber\frac{d^2}{dt^2}\Big|_{t=0}\lambda=\int_M \Big\langle \gamma,  \overline{N}_f(\gamma) \Big\rangle e^{-f}dV_g,
\end{align*}

\end{theorem}

In this section, we would like to discuss the local behavior of steady generalized Ricci solitons. Recall that in \cite{KK} section 4, we define an elliptic operator $\overline{L}_f$ by 
\begin{align}
   \overline{L}_f(\gamma)\coloneqq\frac{1}{2}\overline{\Delta}_f \gamma+\mathring{R}^+(\gamma). \label{L}
\end{align}

\begin{defn} \label{IGSD}
Let $\mathcal{G}(g,b)$ be a steady gradient generalized Ricci soliton. 
\begin{itemize}
    \item  A 2-tensor $\gamma\in \otimes^2 T^*M$ is called an \emph{infinitesimal generalized solitonic deformation} of $\mathcal{G}$ if $\overline{N}_f(\gamma)=0$.
    \item A 2-tensor $\gamma\in \otimes^2 T^*M$ is called an \emph{essential infinitesimal generalized solitonic deformation} of $\mathcal{G}$ if $\overline{L}_f(\gamma)=0$ and $\gamma\in \ker\overline{\divg}_f$.
\end{itemize}
where $\overline{N}_f$ and $\overline{L}_f$ are operators defined by equations (\ref{N}) and (\ref{L}) respectively. In the following, we focus on essential infinitesimal generalized solitonic deformation and denote the set of all essential infinitesimal generalized solitonic deformations by $IGSD$, i.e.
\begin{align*}
    IGSD=\{\gamma\in\otimes^2 T^*M: \overline{L}_f(\gamma)=0 \text{ and $\gamma\in\ker\overline{\divg}_f$.}\}.
\end{align*}
\end{defn}

In the Bismut flat case, we further derive that 
\begin{proposition}[\cite{KK} Proposition 5.1.]\label{BFG}
    Let $M$ be a compact manifold, $E\cong (TM\oplus T^*M)_{H_0}$ be an exact Courant algebroid with a background closed 3-form $H_0$ and $\mathcal{G}(g,b)$ is a Bismut-flat metric. The following statements are equivalent.
    \begin{enumerate}
        \item\label{a} $\gamma$ is an essential infinitesimal generalized solitonic deformation of $\mathcal{G}$.
        \item $\overline{\nabla}\gamma=0$ where $\overline{\nabla}$ is given in (\ref{MBC}).
        \item\label{b} Let $h$ be the symmetric part of $\gamma$ and $-K$ be the anti-symmetric part of $\gamma$. Then,
        \begin{align}
            \nabla_mh_{ij}=-\frac{1}{2}(H_{mik}K_{jk}+H_{mjk}K_{ik}),\quad\nabla_mK_{ij}=-\frac{1}{2}(H_{mjk}h_{ik}-H_{mik}h_{jk}) \label{long}
        \end{align}
        \item $\frac{\partial^2}{\partial t^2}\lambda(\gamma)=0.$
     \end{enumerate}

\end{proposition}

\subsection{Commutator Formulas}

In this subsection, we are going to derive some useful formulas that can help us simplify the second variation operator. To get some motivation of calculation below, one can consult with the shrinking Ricci soliton case in \cite{C,C2,C3}.
\begin{lemma}
Suppose $\mathcal{G}(g,b)$ is a steady gradient generalized Ricci soliton on a smooth compact manifold $M$. For any $(u,v)\in T^*M\times T^*M$, 
\begin{align}
    \overline{\divg}_f\overline{\divg}_f^*(u,v)=\Big(-\Delta^+_f u-\nabla(\divg_fv),-\Delta^-_f v-\nabla(\divg_fu)\Big). \label{Comm1}
\end{align}
\end{lemma}
\begin{proof}
Suppose $\gamma=\overline{\divg}_f^*(u,v)$, then 
\begin{align*}
    \gamma_{ij}&=-(\nabla^+)^i u_j-(\nabla^-)^jv_i
    \\&=-\nabla^i u_j-\nabla^jv_i+\frac{1}{2}H_{ijk}(u_k+v_k).
\end{align*}
 From (\ref{PM}) and (\ref{bar1}), we compute that 
\begin{align*}
    (\nabla^+)^m\gamma_{ml}-\nabla_mf \gamma_{ml}&=\nabla^m\gamma_{ml}-\frac{1}{2}H_{mlk}\gamma_{mk}-\nabla_mf \gamma_{ml}
    \\&=\nabla^m(-\nabla^m u_l-\nabla^lv_m+\frac{1}{2}H_{mlk}(u_k+v_k))-\frac{1}{2}H_{mlk}(-\nabla^m u_k-\nabla^kv_m+\frac{1}{2}H_{mkt}(u_t+v_t))
    \\& \kern2em -\nabla_mf (-\nabla^m u_l-\nabla^lv_m+\frac{1}{2}H_{mlk}(u_k+v_k))
    \\&=-\Delta_f u_l-\nabla_m\nabla_lv_m+\nabla_m f \nabla_lv_m-\frac{1}{2}(d_f^*H)_{lk}(u_k+v_k)+\frac{1}{4}H^2_{lt}(u_t+v_t)+H_{mlk}\nabla_mu_k
    \\&= -\Delta_f u_l-\nabla_l(\divg_f v+\nabla_m fv_m)-R_{lt}v_t+\nabla_m f \nabla_lv_m+\frac{1}{4}H^2_{lt}(u_t+v_t)+H_{mlk}\nabla_mu_k
    \\& =-\Delta_f u_l-\nabla_l(\divg_f v)+H_{mlk}\nabla_mu_k+\frac{1}{4}H^2_{lt}u_t
    \\&= -\Delta_f^+ u_l-\nabla_l(\divg_f v),
\end{align*}
where we use the fact that 
\begin{align*}
    \Delta_f^+ u_l=\Delta_fu_l+H_{ijl}\nabla_iu_j-\frac{1}{4}H^2_{kl}u_k.
\end{align*}
The other part follows similarly.

\end{proof}

\begin{lemma}
Suppose $\mathcal{G}(g,b)$ is a steady gradient generalized Ricci soliton on a smooth compact manifold $M$. For any $(u,v)\in T^*M\times T^*M$,
\begin{align}
\overline{\divg}_f\overline{\divg}_f\overline{\divg}_f^*(u,v)=-\Delta_f(\divg_fu+\divg_fv). \label{Comm2}
\end{align}
\begin{proof}
First, we note that for any $a\in C^\infty(M)$,    
\begin{align*}
    \nabla_l\Delta_f a&=\nabla_l(\Delta a-\nabla_i a\nabla_i f)
    \\&=\Delta(\nabla_l a)-R_{li}\nabla_i a -\nabla_l\nabla_i f\nabla_i a-\nabla_i f\nabla_l\nabla_i a
    \\&= \Delta_f(\nabla_l a)-\frac{1}{4}H_{li}^2\nabla_i a.
\end{align*}
Given $\omega\in T^*M$,
\begin{align*}
    \int a \divg_f \Delta_f\omega e^{-f}dV_g &=\int -\langle \Delta_f\nabla a,\omega \rangle e^{-f}dV_g
    \\&=\int -\langle \nabla\Delta_f a,\omega \rangle-\frac{1}{4}H_{li}^2\nabla_i a\omega_l e^{-f}dV_g
    \\& =\int a \big(\Delta_f\divg_f\omega+\frac{1}{4}\divg_f(H^2_{il}\omega_i) )\big) e^{-f}dV_g,
\end{align*}
for any $a\in C^\infty(M)$. Note that 
\begin{align*}
    \divg_f(\Delta_f^+\omega)&=\divg_f(\Delta_f\omega_l+H_{ijl}\nabla_i\omega_j-\frac{1}{4}H^2_{kl}\omega_k)
    \\&=\divg_f(\Delta_f\omega)+(d_f^*H)_{ij}\nabla_i\omega_j+H_{ijl}\nabla_l\nabla_i\omega_j-\frac{1}{4}\divg_f(H^2_{il}\omega_i)    \\&=\divg_f(\Delta_f\omega)+H_{ijl}\nabla_l\nabla_i\omega_j-\frac{1}{4}\divg_f(H^2_{il}\omega_l).
\end{align*}
Besides, 
\begin{align*}    
H_{ijl}\nabla_l\nabla_i\omega_j&=H_{ijl}(\nabla_i\nabla_l\omega_j-R_{lijk}\omega_k)
\\&=- H_{ijl}\nabla_l\nabla_i\omega_j,
\end{align*}
where we use the Bianchi identity to see that $H_{ijl}R_{lijk}\omega_k=0$. Therefore, 
\begin{align*}
    \divg_f(\Delta_f^+ \omega)&=\divg_f(\Delta_f\omega)-\frac{1}{4}\divg_f(H^2_{il}\omega_i)=\Delta_f(\divg_f\omega).
\end{align*}
Similarly, $\divg_f(\Delta_f^- \omega)=\Delta_f(\divg_f\omega)$. Then, apply (\ref{Comm1}), we finally conclude that 
\begin{align*}
\overline{\divg}_f\overline{\divg}_f\overline{\divg}_f^*(u,v)&=\overline{\divg}_f\Big(-\Delta^+_f u_l-\nabla_l(\divg_fv),-\Delta^-_f v_l-\nabla_l(\divg_fu)\Big)
\\&=-\Delta_f(\divg_fu+\divg_fv).
\end{align*}

\end{proof}
  
\end{lemma}

\begin{remark}
 In the above proof, we deduce that for any one form $\omega\in T^*M$, 
 \begin{align*}
     \divg_f(\Delta^+_f\omega)=\Delta_f(\divg_f\omega).
 \end{align*}
Define an operator $\Phi$ by 
\begin{align}
        \nonumber \Phi: \quad T^*M&\times T^*M\longrightarrow T^*M\times T^*M
        \\& (u,v)\longmapsto (\Delta^+_f u, \Delta_f^-v). \label{37}
\end{align}
We are able to say that 
\begin{align*}
    \Delta_f\big(\kern0.25em\overline{\divg}_f(u,v)\big)=\overline{\divg}_f \big(\Phi(u,v)\big).
\end{align*}
\end{remark}

By (\ref{Comm1}) and (\ref{Comm2}), we derive the following proposition.
\begin{proposition}\label{Comm5}
Suppose $\mathcal{G}(g,b)$ is a steady gradient generalized Ricci soliton on a smooth compact manifold $M$. For any 2-tensor $\gamma=\overline{\divg}_f^*(u,v)$ with $(u,v)\in T^*M\times T^*M$, $\phi=-\divg_fu-\divg_fv$ is the unique solution of 
\begin{align*}
    \Delta_f \phi=\overline{\divg}_f\overline{\divg}_f\gamma,\quad \int_M \phi e^{-f}dV_g=0.
\end{align*}   
\end{proposition}

Next, we compute $\overline{\Delta}_f\gamma+2\mathring{R}^+(\gamma)$ where $\gamma=\overline{\divg}^*_f(u,v)$ with $(u,v)\in T^*M\times T^*M$.

\begin{lemma}\label{Comm3}
Suppose $\mathcal{G}(g,b)$ is a steady gradient generalized Ricci soliton on a smooth compact manifold $M$. For any $u\in T^*M$, \begin{align*}
 \Delta_f\nabla_iu_j&=\nabla_i^+\Delta^+_fu_j+\frac{1}{2}H_{ijk}\Delta_fu_k+\frac{1}{2}H_{ijk}H_{lmk}\nabla_lu_m-\frac{1}{8}H_{ijk}H_{kl}^2u_l-\nabla_iH_{klj}\nabla_ku_l-H_{klj}\nabla_i\nabla_ku_l
    \\&\kern2em+\frac{1}{4}(\nabla_iH_{jl}^2+\nabla_jH_{il}^2-\nabla_lH_{ij}^2)u_l+\frac{1}{4}(H^2_{jl}\nabla_iu_l+H^2_{il}\nabla_lu_j)-2R_{kijl}\nabla_ku_l
\end{align*}   
\end{lemma}
\begin{proof}
Using the commutator formula, we compute
\begin{align*}
    \Delta_f\nabla_i u_j&=\nabla_k\nabla_k(\nabla_iu_j)-\nabla_k f\nabla_k\nabla_iu_j
    \\&= \nabla_k(\nabla_i\nabla_ku_j-R_{kijl}u_l)-\nabla_k f(\nabla_i\nabla_ku_j-R_{kijl}u_l)
    \\&=\nabla_i\Delta u_j-R_{kikl}\nabla_lu_j-2R_{kijl}\nabla_ku_l-\nabla_kR_{kijl}u_l-\nabla_k f(\nabla_i\nabla_ku_j-R_{kijl}u_l)
    \\&=\nabla_i\Delta_fu_j+\frac{1}{4}H^2_{il}\nabla_lu_j-2R_{kijl}\nabla_ku_l-(\divg_f\text{Rm})_{ijl}u_l
    \\&=\nabla_i^+\Delta^+_fu_j+\frac{1}{2}H_{ijk}\Delta_f^+u_k-\nabla_iH_{klj}\nabla_ku_l-H_{klj}\nabla_i\nabla_ku_l+\frac{1}{4}\nabla_iH_{jl}^2u_l+\frac{1}{4}(H^2_{jl}\nabla_iu_l+H^2_{il}\nabla_lu_j)
    \\&\kern2em-2R_{kijl}\nabla_ku_l-(\divg_f\text{Rm})_{ijl}u_l
    \\&=\nabla_i^+\Delta^+_fu_j+\frac{1}{2}H_{ijk}\Delta_fu_k+\frac{1}{2}H_{ijk}H_{lmk}\nabla_lu_m-\frac{1}{8}H_{ijk}H_{kl}^2u_l-\nabla_iH_{klj}\nabla_ku_l-H_{klj}\nabla_i\nabla_ku_l+\frac{1}{4}\nabla_iH_{jl}^2u_l
    \\&\kern2em+\frac{1}{4}(H^2_{jl}\nabla_iu_l+H^2_{il}\nabla_lu_j)-2R_{kijl}\nabla_ku_l-(\divg_f\text{Rm})_{ijl}u_l.
\end{align*}
We notice that
\begin{align*}
    (\divg_f\text{Rm})_{ijl}&=\nabla_kR_{kijl}-\nabla_k f R_{kijl}
    \\&= \nabla_l R_{ij}-\nabla_j R_{il}-\nabla_k f R_{kijl}
    \\&=\frac{1}{4}\nabla_l H^2_{ij}-\frac{1}{4}\nabla_jH^2_{il}.
\end{align*}    
Therefore, we finish the proof of the lemma.

\end{proof}

\begin{lemma}\label{Comm4}
Suppose $\mathcal{G}(g,b)$ is a steady gradient generalized Ricci soliton on a smooth compact manifold $M$. For any $u\in T^*M$, 
\begin{align*}
    \Delta_f(H_{ijk}u_k)&=-2(R_{lijm}H_{lmk}+R_{ljkm}H_{lmi}+R_{lkim}H_{lmj})u_k+\frac{1}{4}(H^2_{im} H_{mjk}-H^2_{jm}H_{mik}+H^2_{km}H_{mij})u_k
    \\&\kern2em +H_{ijk}\Delta_fu_k+2\nabla_lH_{ijk}\nabla_lu_k
\end{align*}    
\end{lemma}
\begin{proof}
We first compute $\Delta_f H$,    
\begin{align*}
    \Delta_f(H_{ijk})&=\nabla_l\nabla_l H_{ijk}-\nabla_l f \nabla_lH_{ijk}
    \\&=\nabla_l(\nabla_iH_{ljk}-\nabla_jH_{lik}+\nabla_kH_{lij})-\nabla_l f \nabla_lH_{ijk}
    \\&=-\nabla_i(d^*H)_{jk}+\nabla_j(d^*H)_{ik}-\nabla_k(d^*H)_{ij}+R_{im}H_{mjk}-R_{jm}H_{mik}+R_{km}H_{mij}-\nabla_l f \nabla_lH_{ijk}
    \\&\kern2em -R_{lijm}H_{lmk}+R_{ljim}H_{lmk}-R_{lkim}H_{lmj}-R_{likm}H_{ljm}+R_{ljkm}H_{lim}-R_{lkjm}H_{lim}
    \\&=\frac{1}{4}(H^2_{im} H_{mjk}-H^2_{jm}H_{mik}+H^2_{km}H_{mij})-2R_{lijm}H_{lmk}-2R_{ljkm}H_{lmi}-2R_{lkim}H_{lmj}
\end{align*}
where the last step we use the assumption that $\Rc^{H,f}=0$ and $d^*_fH=0$. Therefore,
\begin{align*}
    \Delta_f(H_{ijk}u_k)&=\Delta_f(H_{ijk})u_k+H_{ijk}\Delta_fu_k+2\nabla_lH_{ijk}\nabla_lu_k
    \\&=-2(R_{lijm}H_{lmk}+R_{ljkm}H_{lmi}+R_{lkim}H_{lmj})u_k+\frac{1}{4}(H^2_{im} H_{mjk}-H^2_{jm}H_{mik}+H^2_{km}H_{mij})u_k
    \\&\kern2em +H_{ijk}\Delta_fu_k+2\nabla_lH_{ijk}\nabla_lu_k
\end{align*}
\end{proof}

\begin{proposition}\label{Comm6}
Suppose $\mathcal{G}(g,b)$ is a steady gradient generalized Ricci soliton on a smooth compact manifold $M$. For any 2-tensor $\gamma=\overline{\divg}^*_f(u,v)$ with $(u,v)\in T^*M\times T^*M$, we have
\begin{align*}
   \overline{L}_f\big(\overline{\divg}^*_f(u,v)\big)=\frac{1}{2}\overline{\divg}_f^*\Phi(u,v).
\end{align*}
where $\overline{L}_f$ is defined at (\ref{L}) and $\Phi$ is defined at (\ref{37}).
\end{proposition}    
\begin{proof}
    Let us first focus on the $u$-part. We compute
\begin{align*}
    &-\overline{\Delta}_f(\nabla_i^+ u_j)-2R^+_{kijl}(\nabla_k^+u_l)
    \\&\kern2em=-\Delta_f(\nabla_i^+ u_j)+H_{mjk}\nabla_m(\nabla_i^+ u_k)-H_{mik}\nabla_m(\nabla_k^+ u_j)+\frac{1}{4}(H^2_{jl}\nabla_i^+ u_l+H^2_{il}\nabla_l^+ u_j)+\frac{1}{2}H_{mkj}H_{mli}(\nabla_l^+ u_k)
    \\&\kern4em -2(R_{kijl}+\frac{1}{2}\nabla_kH_{ijl}-\frac{1}{2}\nabla_iH_{kjl}-\frac{1}{4}H_{klm}H_{ijm}+\frac{1}{4}H_{kjm}H_{ilm})(\nabla_k^+ u_l)
    \\&\kern2em=-\Delta_f(\nabla_i^+ u_j)-2R_{kijl}(\nabla_k^+ u_l)+H_{mjk}\nabla_m(\nabla_i^+ u_k)-H_{mik}\nabla_m(\nabla_k^+ u_j)+\frac{1}{4}(H^2_{jl}\nabla_i^+ u_l+H^2_{il}\nabla_l^+ u_j)
    \\&\kern4em-(\nabla_kH_{ijl}-\nabla_iH_{kjl})(\nabla_k^+ u_l)+\frac{1}{2}H_{klm}H_{ijm}(\nabla_k^+ u_l)+\frac{1}{2}H_{mkj}H_{mli}(\nabla_k^+ u_l+\nabla_l^+u_k).
\end{align*}    
Here, we compute
\begin{align*}
    H_{mjk}\nabla_m(\nabla_i^+ u_k)-H_{mik}\nabla_m(\nabla_k^+ u_j)&=H_{mjk}\nabla_m(\nabla_i u_k-\frac{1}{2}H_{ikl}u_l)-H_{mik}\nabla_m(\nabla_k u_j-\frac{1}{2}H_{kjl}u_l)
    \\&=H_{mjk}\nabla_m\nabla_iu_k-H_{mik}\nabla_m\nabla_ku_j-\frac{1}{2}H_{mjk}\nabla_m H_{ikl}u_l+\frac{1}{2}H_{mik}\nabla_m H_{kjl}u_l
    \\&\kern2em -\frac{1}{2}H_{mjk}H_{ikl}(\nabla_m u_l+\nabla_lu_m).
\end{align*}
Note that 
\begin{align*}
    H_{mik}\nabla_m\nabla_ku_j&=H_{mik}(\nabla_k\nabla_mu_j-R_{mkjl}u_l)
    \\&=-H_{mik}\nabla_m\nabla_ku_j+R_{mkjl}H_{mki}u_l.
\end{align*}
So, 
\begin{align*}
    H_{mik}\nabla_m\nabla_ku_j=\frac{1}{2}R_{mkjl}H_{mki}u_l=-R_{mlkj}H_{mki}u_l.
\end{align*}
Using the fact that $dH=0$, we have
\begin{align*}
    H_{mjk}\nabla_m H_{ikl}=H_{mjk}(\nabla_iH_{klm}-\nabla_kH_{ilm}+\nabla_lH_{ikm})
\end{align*}
then
\begin{align*}
    H_{mjk}\nabla_mH_{ikl}+H_{mik}\nabla_mH_{jkl}&=\frac{1}{2}(H_{mjk}\nabla_iH_{klm}+H_{mik}\nabla_jH_{klm}+\nabla_lH^2_{ij})
    \\&=\frac{1}{2}\big(-\nabla_iH^2_{jl}-\nabla_jH^2_{il}+\nabla_lH^2_{ij}-(\nabla_iH_{mjk}+\nabla_jH_{mik})H_{klm}\big).
\end{align*}
Therefore,
\begin{align*}
    H_{mjk}\nabla_m(\nabla_i^+ u_k)-H_{mik}\nabla_m(\nabla_k^+ u_j)&= H_{mjk}\nabla_i\nabla_m u_k-R_{mikl}H_{mjk}u_l+R_{mlkj}H_{mki}u_l
    \\&\kern2em +\frac{1}{4}(\nabla_iH^2_{jl}+\nabla_jH^2_{il}-\nabla_lH^2_{ij}+(\nabla_iH_{mjk}+\nabla_jH_{mik})H_{klm}\big)u_l
    \\& \kern2em -\frac{1}{2}H_{mjk}H_{ikl}(\nabla_m u_l+\nabla_lu_m).
\end{align*}
Next, we have
\begin{align*}
    \frac{1}{4}(H^2_{jl}\nabla_i^+ u_l+H^2_{il}\nabla_l^+ u_j)=\frac{1}{4}(H^2_{jl}\nabla_i u_l+H^2_{il}\nabla_l u_j)-\frac{1}{8}(H^2_{jl}H_{ilk}+H^2_{il}H_{ljk})u_k,
\end{align*}
\begin{align*}
    \frac{1}{2}H_{klm}H_{ijm}(\nabla_k^+u_l)=\frac{1}{2}H_{klm}H_{ijm}(\nabla_ku_l)-\frac{1}{4}H^2_{kl}H_{ijk}u_l.
\end{align*}
We write
\begin{align*}
    \nabla^+_k u_l&=\nabla_ku_l-\frac{1}{2}H_{klm}u_m
    \\&=\frac{1}{2}(\nabla_ku_l+\nabla_lu_k)+\frac{1}{2}(\nabla_ku_l-\nabla_lu_k)-\frac{1}{2}H_{klm}u_m,
\end{align*}
and then
\begin{align*}
    -(\nabla_kH_{ijl}-\nabla_iH_{kjl})(\nabla_k^+u_l)&=-\frac{1}{2}\nabla_kH_{ijl}(\nabla_ku_l+\nabla_lu_k)+\frac{1}{4}(\nabla_{i}H_{kjl}+\nabla_jH_{kil})(\nabla_ku_l-\nabla_lu_k)
    \\& \kern2em-\frac{1}{4}(\nabla_{i}H_{kjl}+\nabla_jH_{kil})H_{klm}u_m.
\end{align*}
In summary, by using \Cref{Comm3} and \Cref{Comm4}, we get
\begin{align*}
    &-\overline{\Delta}_f(\nabla_i^+ u_j)-2R^+_{kijl}(\nabla_k^+u_l)
    \\&\kern2em=-\Delta_f(\nabla^+_iu_j)-2R_{kijl}\nabla^+_ku_l+H_{mjk}\nabla_i\nabla_m u_k-R_{mikl}H_{mjk}u_l+R_{mlkj}H_{mki}u_l
    \\&\kern4em +\frac{1}{4}(\nabla_iH^2_{jl}+\nabla_jH^2_{il}-\nabla_lH^2_{ij}+(\nabla_iH_{mjk}+\nabla_jH_{mik})H_{klm}\big)u_l
    \\&\kern4em+\frac{1}{4}(H^2_{jl}\nabla_i u_l+H^2_{il}\nabla_l u_j)-\frac{1}{8}(H^2_{jl}H_{ilk}+H^2_{il}H_{ljk})u_k+\frac{1}{2}H_{klm}H_{ijm}(\nabla_ku_l)-\frac{1}{4}H^2_{kl}H_{ijk}u_l
    \\&\kern4em-\frac{1}{2}\nabla_kH_{ijl}(\nabla_ku_l+\nabla_lu_k)+\frac{1}{4}(\nabla_{i}H_{kjl}+\nabla_jH_{kil})(\nabla_ku_l-\nabla_lu_k)-\frac{1}{4}(\nabla_{i}H_{kjl}+\nabla_jH_{kil})H_{klm}u_m
    \\& \kern2em=-\nabla_i^+\Delta_f^+u_j+\frac{1}{2}(\nabla_iH_{klj}+\nabla_jH_{kil}-\nabla_lH_{ijk})(\nabla_ku_l)+\frac{1}{2}\nabla_lH_{ijk}\nabla_lu_k
    \\&\kern2em=-\nabla_i^+\Delta_f^+u_j.
\end{align*}
$v-$part follows similarly, so we skip the detail here.
\end{proof}

\begin{theorem}\label{MT1}
Suppose $\mathcal{G}(g,b)$ is a steady gradient generalized Ricci soliton on a smooth compact manifold $M$. For any 2-tensor $\gamma=\overline{\divg}^*_f(u,v)$ with $(u,v)\in T^*M\times T^*M$, we have 
\begin{align*}
    \overline{N}_f(\gamma)=0.
\end{align*}
Hence, the second variation formula of $\lambda$ at $\mathcal{G}(g,b)$ reduces to 
\begin{align*}
     \nonumber\frac{d^2}{dt^2}\Big|_{t=0}\lambda=\int_M \Big\langle \gamma, \frac{1}{2}\overline{\Delta}_f\gamma+\mathring{R}^+(\gamma) \Big\rangle e^{-f}dV_g,
\end{align*} 
where $\gamma\in\ker\overline{\divg}$.
\end{theorem}
\begin{proof}
Recall that (\ref{N}), $\overline{N}_f(\gamma)=\frac{1}{2}\overline{\Delta}_f\gamma+\mathring{R}^+(\gamma)+\frac{1}{2}\overline{\divg}_f^*\overline{\divg}_f\gamma+\frac{1}{2}(\nabla^+)^2\phi$. Due to \Cref{Comm5}, \Cref{Comm6} and (\ref{Comm1}), we have
\begin{align*}
    \overline{N}_f(\gamma)&=\frac{1}{2}\overline{\divg}_f^*(\Delta_f^+ u, \Delta^-_f v)+\frac{1}{2}\overline{\divg}_f^*\overline{\divg}_f\overline{\divg}_f^*(u,v)-\frac{1}{2}(\nabla^+)^2(\divg_fu+\divg_fv)
    \\&=\frac{1}{2}\overline{\divg}_f^*(\Delta_f^+ u, \Delta^-_f v)+\frac{1}{2}\overline{\divg}_f^*\big(-\Delta^+_f u-\nabla(\divg_fv),-\Delta^-_f v-\nabla(\divg_fu)\big)-\frac{1}{2}(\nabla^+)^2(\divg_fu+\divg_fv)
    \\&=0.
\end{align*}

\end{proof}

Using integration by parts, we also derive the following result.
\begin{corollary}\label{Comm7}
Suppose $\mathcal{G}(g,b)$ is a steady gradient generalized Ricci soliton on a smooth compact manifold $M$. For any 2-tensor $\gamma$, we get
\begin{align*}
    \overline{\divg}_f(\overline{L}_f(\gamma))=\frac{1}{2}\Phi(\overline{\divg}_f\gamma)
\end{align*}
\end{corollary}

\begin{proof}
Given a 2-tensor $\gamma$, for any $(u,v)\in T^*M\times T^*M$,
\begin{align*}
    \int \Big\langle \gamma, \overline{L}_f\big(\overline{\divg}^*_f(u,v)\big) \Big\rangle e^{-f}dV_g&=\int \Big\langle \gamma, \frac{1}{2}\overline{\divg}_f^*\Phi(u,v) \Big\rangle e^{-f}dV_g
    \\&=\int \Big\langle \frac{1}{2}\Phi(\overline{\divg}_f\gamma),(u,v)  \Big\rangle e^{-f}dV_g
    \\&=\int \Big\langle \overline{\divg}_f(\overline{L}_f(\gamma)),(u,v)  \Big\rangle e^{-f}dV_g.
\end{align*}

\end{proof}

By \Cref{Comm7} and \Cref{Comm6}, we further derive the following.
\begin{corollary}
Suppose $\mathcal{G}(g,b)$ is a steady gradient generalized Ricci soliton on a smooth compact manifold $M$. The operator $\overline{L}_f$ preserves the decomposition $\otimes^2T^*M=\ker\overline{\divg}_f\oplus \text{\emph{Im}}\kern0.25em\overline{\divg}_f^*.$      
\end{corollary}

\section{Local Maximum of the generalized Einstein--Hilbert functional}

In this section, we start to discuss the local maximum property of generalized Einstein--Hilbert functional. Recall that the similar work was done by Hasholfer \cite{Has2,Has1}, Sesum \cite{Ses} and Kröncke \cite{Kr,Kr2} in Einstein and Ricci soliton case.
To generalize their work, let us first improve the generalized slice theorem (\Cref{slice}).

\subsection{Affine generalized slice theorem}

Let $E$ be an exact Courant algebroid and $\mathcal{G}(g,b)$ is a generalized metric. The generalized slice theorem (\Cref{slice}) suggests that locally we can find a submanifold $S^f_{\mathcal{G}}$ with its tangent space $T_{\mathcal{G}}S^f_{\mathcal{G}}=\ker\overline{\divg}_f$. In this subsection, we further construct an affine slice. First, let us fix some notations.
\begin{itemize}
    \item The automorphism group is given by 
\begin{align*}
    \GDiff_H=\{(\varphi,B)\in\Diff(M)\ltimes\Omega^2: \varphi^*H=H-dB\},
\end{align*}
and its Lie algebra is  
\begin{align*}
      \mathfrak{gdiff}_H=\{(X,\kappa)\in TM\times\Omega^2 : d(i_XH+\kappa)=0\}.
\end{align*}
In particular, we will consider the following Lie subalgebra 
\begin{align*}
       \mathfrak{gdiff}_H^e=\{(X,\kappa)\in TM\times \Omega^2 : i_XH+\kappa=d\alpha \text{  for some $\alpha\in\Omega^1$}\}
\end{align*}
and denote its integration by $\GDiff^e_H$. Note that for any $(X,\alpha)\in TM\times T^*M$, we can define $\kappa=d\alpha-\iota_XH$ so that $(X,\kappa)\in \mathfrak{gdiff}_H^e $. In other words, we may also say $\mathfrak{gdiff}_H^e\cong TM\times T^*M$.

\item  We adopt an isomorphism 
\begin{align*}
    \phi: TM&\times T^*M \longrightarrow T^*M\times T^*M
    \\& (X,\alpha)\kern0.8em\longmapsto \kern0.5em (u,v)
\end{align*}
where $u^\sharp+v^\sharp=-2X$ and $u-v=-2\alpha$. 
\item Recall that $\overline{\divg}_f^*: T^*M\times T^*M \longrightarrow \otimes^2 T^*M$ and $T^*M\times T^*M=\ker\overline{\divg}_f^*\oplus (\ker\overline{\divg}_f^*)^\perp$, we denote $\Pi:T^*M\times T^*M \longrightarrow \ker\overline{\divg}_f^*$ to be the orthogonal projection to the subspace $\ker\overline{\divg}_f^*$.
\end{itemize}
Define 
\begin{align*}
    \mathcal{A}:\kern1em  & T^*M\times T^*M\times \mathcal{GM} \longrightarrow T^*M\times T^*M
    \\& \kern1em\big( (u,v),\widetilde{\mathcal{G}}\big)\longmapsto \overline{\divg}_f\big((\varphi_X,B)\cdot\widetilde{\mathcal{G}}-\mathcal{G} \big)+\Pi(u,v),
\end{align*}
where $(\varphi_X,B)\in \GDiff^e_H$ is the flow generated by $\phi^{-1}(u,v)\in \mathfrak{gdiff}_H^e$ at time 1. W.L.O.G, we may take $\mathcal{G}=\mathcal{G}(g,0)$, then it is clear that $\mathcal{A}(0,0,\mathcal{G})=0$. By (\ref{GA}) and the \Cref{Gamma}, we write
    \begin{align*}
        d\mathcal{A}\Big|_{(0,0,\mathcal{G})}\big( (u,v),0\big)=\overline{\divg}_f (\mathcal{L}_Yg+\kappa)+\Pi(u,v),
    \end{align*}
    where $\iota_YH+\kappa=d\beta$ and $(Y,\beta)=\phi^{-1}(u,v)$. For convenience, we define $\mathcal{B}: T^*M\times T^*M\longrightarrow T^*M\times T^*M$ by $\mathcal{B}(u,v)=d\mathcal{A}\Big|_{(0,0,\mathcal{G})}\big( (u,v),0\big)$

\begin{lemma}
   $ \mathcal{B}$ is injective.
\end{lemma}
\begin{proof}
 Suppose that $\mathcal{B}(u,v)=(0.0)$. Since $\im \overline{\divg}_f\cap \ker\overline{\divg}_f^*=\{0\}$, we deduce that $\overline{\divg}_f(\mathcal{L}_Yg+\kappa)=\Pi(u,v)=0$, where $\kappa=d\beta-\iota_YH$ and $(Y,\beta)=\phi^{-1}(u,v)$. For any $(x,y)\in T^*M\times T^*M$, we get
\begin{align*}
    0=\int_M \big\langle \overline{\divg}_f(\mathcal{L}_Yg+\kappa), (x,y) \big\rangle e^{-f}dV_g=\int_M \big\langle \mathcal{L}_Yg+\kappa,\overline{\divg}_f^* (x,y) \big\rangle e^{-f}dV_g.
\end{align*}
Rewrite (\ref{dbars}), we know 
\begin{align*}
    \overline{\divg}_f^*(x,y)_{ij}=-\frac{1}{2}(\mathcal{L}_{x^\sharp }g+\mathcal{L}_{y^\sharp }g)_{ij}-\frac{1}{2}d(x-y)_{ij}+\frac{1}{2}H_{ijk}(x+y)_k.
\end{align*}
In particular, when $(x,y)=\phi(Y,\beta)$
\begin{align*}
     \overline{\divg}_f^*(x,y)=\mathcal{L}_Yg+\kappa.
\end{align*}
Therefore, 
\begin{align*}
    0&=\int_M \big\langle \mathcal{L}_Yg+\kappa,\overline{\divg}_f^* (u,v) \big\rangle e^{-f}dV_g
    \\&=\int_M(\mathcal{L}_Yg+\kappa)^2e^{-f}dV_g.
\end{align*}
and see that $\mathcal{L}_Yg+\kappa=0$. In other words, $(u,v)\in \ker\overline{\divg}_f^*$. Since $\Pi(u,v)=0$, we could further see that $(u,v)=(0,0)$.
\end{proof}

\begin{lemma}
      $ \mathcal{B}$ is surjective. 
\end{lemma}
\begin{proof}
    For any $(x,y),(u,v)\in T^*M\times T^*M$,
    \begin{align*}
        \int_M \langle \mathcal{B}(u,v),(x,y) \rangle e^{-f}dV_g&=  \int_M \Big[\big\langle \mathcal{L}_Yg+\kappa, \overline{\divg}_f^*(x,y)\big\rangle +\big\langle \Pi(u,v), \Pi(x,y) \big\rangle \Big]e^{-f}dV_g
        \\&=\int_M \Big[\big\langle \mathcal{L}_Yg+\kappa,  \mathcal{L}_Zg+\zeta\big\rangle +\big\langle \Pi(u,v), \Pi(x,y) \big\rangle \Big]e^{-f}dV_g
        \\&=\int_M \langle (u,v), \mathcal{B}(x,y) \rangle e^{-f}dV_g.
    \end{align*}
Here, we denote $\kappa=d\beta-\iota_YH$, $(Y,\beta)=\phi^{-1}(u,v)$ and $\zeta=d\delta-\iota_ZH$, $(Z,\delta)=\phi^{-1}(x,y)$. Thus, we see that $\mathcal{B}$ is self-adjoint and $\mathcal{B}$ is surjective.   
    
\end{proof}

Using the implicit function theorem, we can conclude the following affine generalized slice theorem.
\begin{theorem}\label{affineslice}
    Let $\mathcal{G}$ be a generalized metric on an exact Courant algebroid $E$ and let $f$ be $\Isom(\mathcal{G})$ invariant. There exists a neighborhood $\mathcal{U}$ of $\mathcal{G}$ such that for any $\widetilde{\mathcal{G}}$, we can find $(X,\kappa)\in \mathfrak{gdiff}_H^e$ so that
    \begin{align*}
        \overline{\divg}_f\big( (\varphi_X,B)\cdot\widetilde{ \mathcal{G}}- \mathcal{G} \big)=0,
    \end{align*}
    where $(\varphi_X,B)\in \GDiff_H^e$ is the flow generated by $(X,\kappa)$ at time 1.
\end{theorem}
\begin{proof}
    Since $\mathcal{B}$ is invertible, by the implicit function theorem, there exists a neighborhood $\mathcal{U}$ such that for any $\widetilde{\mathcal{G}}\in\mathcal{U}$, we can find $(u,v)$ such that 
    \begin{align*}
        0=\mathcal{A}(u,v,\widetilde{\mathcal{G}})=\overline{\divg}_f\big((\varphi_X,B)\cdot\widetilde{\mathcal{G}}-\mathcal{G} \big)+\Pi(u,v),
    \end{align*}
where $(\varphi_X,B)\in \GDiff^e_H$ is the flow generated by $\phi^{-1}(u,v)\in \mathfrak{gdiff}_H^e$ at time 1. Note that $\im \overline{\divg}_f\cap \ker\overline{\divg}_f^*=\{0\}$, we complete the proof.
    
\end{proof}

\subsection{Some Estimates}
In this subsection, let us introduce some estimates that are necessary to prove the main theorem. In the following, $C$ is a constant that may change from line to line.  To start with, let us recall that
\begin{lemma}[\cite{K} Lemma 5.4.]
 Let $M$ be a compact manifold and $E\cong (TM\oplus T^*M)_{H_0}$ be an exact Courant algebroid with a background closed 3-form $H_0$. Given a steady gradient generalized Ricci soliton $\mathcal{G}_0(g_0,b_0)$, there exists a $C^{2,\alpha}$-neighborhood $\mathcal{U}$ of $\mathcal{G}_0$ in $\mathcal{GM}$ such that the minimizer $f$ is uniformly bounded, i.e., there exists a constant $C$ such that 
 \begin{align*}
     \|f\|_{C^{2,\alpha}}\leq C.
 \end{align*}
\end{lemma}

\begin{lemma}[\cite{K} Lemma 5.7.]\label{f1}
 Let $M$ be a compact manifold and $E\cong (TM\oplus T^*M)_{H_0}$ be an exact Courant algebroid with a background closed 3-form $H_0$. Given a steady gradient generalized Ricci soliton $\mathcal{G}_0(g_0,b_0)$, there exists a $C^{2,\alpha}$-neighborhood $\mathcal{U}$ of $\mathcal{G}_0$ in $\mathcal{GM}$ such that for all $\mathcal{G}\in \mathcal{U}$,
 \begin{align*}
     \|\frac{d}{dt}\Big|_{t=0}f_{\mathcal{G}+t\gamma}\|_{C^{2,\alpha}}\leq C\|\gamma\|_{C^{2,\alpha}}, \quad  \|\frac{d}{dt}\Big|_{t=0}f_{\mathcal{G}+t\gamma}\|_{H^i}\leq C\|\gamma\|_{H^i}, \quad i=0,1,2.
 \end{align*}
\end{lemma}

In this subsection, we further deduce the second derivative estimate of the minimizer. 
\begin{lemma}\label{f2}
    Let $M$ be a compact manifold and $E\cong (TM\oplus T^*M)_{H_0}$ be an exact Courant algebroid with a background closed 3-form $H_0$. Given a steady gradient generalized Ricci soliton $\mathcal{G}_0(g_0,b_0)$, there exists a $C^{2,\alpha}$-neighborhood $\mathcal{U}$ of $\mathcal{G}_0$ in $\mathcal{GM}$ such that for all $\mathcal{G}\in \mathcal{U}$,
    \begin{align*}
        \|\frac{\partial^2}{\partial s \partial t}\Big|_{t,s=0} f_{\mathcal{G}+t\gamma+s\delta} \|_{H^i} \leq C\|\gamma\|_{C^{2,\alpha}}\|\delta\|_{H^i}, \quad i=1,2. 
    \end{align*}
\end{lemma}
\begin{proof}
For convenience, we write $\mathcal{G}_s=\mathcal{G}+s\delta$, $()'$ denotes the derivative with respect to the variable $s$ and $\overset{\cdot}{()}$ denotes the derivative with respect to the variable $t$. We recall (\ref{lam}) and the proof of \cite{K} Lemma 3.7, we get 

\begin{align*}
    \frac{\partial}{\partial t} \lambda(\mathcal{G}_s+t\gamma)&=\frac{\partial}{\partial t} \big( R_s-\frac{1}{12}|H_s|^2+2\Delta f_s-|\nabla f_s|^2 \big)
    \\&=-\Delta_{f_s}(\tr_{g_s}\gamma-2\overset{\cdot}{f_s})+\overline{\divg}_{f_s}\overline{\divg}_{f_s}\gamma.
\end{align*}
Thus, 
\begin{align*}
   \frac{\partial^2}{\partial s \partial t}&\Big|_{s,t=0}\lambda(\mathcal{G}+t\gamma+s\delta)=\frac{\partial}{\partial s}\Big|_{s=0}\big( -\Delta_{f_s}(\tr_{g_s}\gamma-2\overset{\cdot}{f_s})+\overline{\divg}_{f_s}\overline{\divg}_{f_s}\gamma \big).
\end{align*}
Let $\ast$ be the Hamilton's notation for a combination of tensor products with contractions. It follows by the direct computation that 
\begin{align} \label{Est1}
     \frac{\partial}{\partial s}\Big|_{s=0}\big( -\Delta_{f_s}(\tr_{g_s}\gamma-2\overset{\cdot}{f_s})\big)&=\Delta_f \Big(  \frac{\partial}{\partial s}\Big|_{s=0} \big(\tr_{g_s}\gamma-2\overset{\cdot}{f_s}\big) \Big)+\delta\ast\nabla^2\gamma+ \nabla\delta \ast\nabla\gamma+\nabla(\tr_{g_s}\gamma-2\overset{\cdot}{f_s})\ast \nabla f_s',
\end{align}
\begin{align} \label{Est2}
     \frac{\partial}{\partial s}\Big|_{s=0}\overline{\divg}_{f_s}\overline{\divg}_{f_s}\gamma &= \frac{\partial}{\partial s}\Big|_{s=0} \big(\divg_{f_s}\divg_{f_s} h-\frac{1}{6}\langle dK,H \rangle \big)
    \nonumber \\&=\delta\ast \nabla^2h+\nabla \delta\ast\nabla h+\nabla^2\delta \ast h+h\ast \nabla^2f_s'+h\ast\nabla f_s\ast \nabla\delta+\nabla f_s'\ast\nabla h+h\ast \nabla f_s\ast \nabla f_s'
     \nonumber\\&\kern2em +\delta\ast \nabla K\ast H+ \nabla\delta\ast K\ast H+ \nabla K\ast \nabla\delta,
\end{align}
where we denote $\gamma=h-K$. By \Cref{f1}, the elliptic regularity and the integration by parts, we further estimate that 
 \begin{align*}
     \| \frac{\partial^2}{\partial s \partial t}\Big|_{s,t=0}f  \|_{H^i}&\leq C \|\Delta(\frac{\partial^2}{\partial s \partial t}\Big|_{s,t=0}f)\|_{H^{i-2}}
     \\&\leq C \|\frac{\partial^2}{\partial s \partial t}\Big|_{s,t=0}\lambda(\mathcal{G}+t\gamma+s\delta) \|_{H^{i-2}}+C\|\gamma\|_{C^{2,\alpha}}\|\delta \|_{H^i} .
 \end{align*}
We left to estimate the second derivative of $\lambda$. Due to the first variation formula of $\lambda$ (\ref{fv}),

\begin{align*}
    \frac{\partial^2}{\partial s \partial t}&\Big|_{s,t=0}\lambda(\mathcal{G}+t\gamma+s\delta)
    \\&=\frac{\partial }{\partial s}\Big|_{s=0}\Big( \int_M -\langle \Rc^{H,f_s}_s,\gamma \rangle_s e^{-f_s}dV_{g_s} \Big)
    \\&=\int_M -\langle (\Rc^{H,f_s}_s)',\gamma\rangle e^{-f_s}dV_g- \int_M  \big(g_s^{-1}\big)'\ast \Rc^{H,f_s}\ast\gamma e^{-f_s}dV_g-\int_M \langle \Rc^{H,f_s},\gamma \rangle(e^{-f_s}dV_{g_s})'.
\end{align*}
As shown in \cite{KK} Proposition 3.6, 
\begin{align*}
   (\Rc^{H,f_s}_s)'=-\frac{1}{2}\overline{\Delta}_{f_s}\delta-\mathring{R}^+(\delta)-\frac{1}{2}\overline{\divg}_{f_s}^*\overline{\divg}_{f_s}\delta-\frac{1}{2}(\nabla^+)^2(\tr_g\delta-2f_s').
\end{align*}
 Using the integration by part and \Cref{f1}, we have
 \begin{align*}
     \Big| \int_M \langle (\Rc^{H,f_s})',\gamma\rangle e^{-f_s}dV_g \Big|\leq C\|\delta\|_{C^{2,\alpha}}\|\gamma\|_{H^2}
 \end{align*}
 Also, it is clear that 
 \begin{align*}
      \Big|\int_M  \big(g_s^{-1}\big)'\ast \Rc^{H,f_s}\ast\gamma e^{-f_s}dV_g \Big|\leq C\|\delta\|_{C^{2,\alpha}}\|\gamma\|_{H^2}
 \end{align*}
 and
 \begin{align*}
     \Big|\int_M \langle \Rc^{H,f_s},\gamma \rangle \frac{d}{ds}\Big|_{s=0}(e^{-f_s}dV_{g_s})\Big|=\Big|\int_M \langle \Rc^{H,f_s},\gamma \rangle \big(\frac{1}{2}\tr_g\delta-f_s' \big) e^{-f_s}dV_g \Big|\leq C\|\delta\|_{C^{2,\alpha}}\|\gamma\|_{H^2}.
 \end{align*}
 Thus, 
 \begin{align*}
     \Big|\frac{\partial^2}{\partial s \partial t}&\Big|_{s,t=0}\lambda(\mathcal{G}+t\gamma+s\delta)\Big|\leq C\|\delta\|_{C^{2,\alpha}}\|\gamma\|_{H^2}.
 \end{align*}
For $H^{i=2}=H^{-1}$, the argument follows similarly from testing with an $H^1$-function.
\end{proof}

To end this subsection, let us introduce the third order estimate of $\lambda$.
\begin{lemma}\label{third}
  Let $M$ be a compact manifold and $E\cong (TM\oplus T^*M)_{H_0}$ be an exact Courant algebroid with a background closed 3-form $H_0$. Given a steady gradient generalized Ricci soliton $\mathcal{G}_0(g_0,b_0)$, there exists a $C^{2,\alpha}$-neighborhood $\mathcal{U}$ of $\mathcal{G}_0$ in $\mathcal{GM}$ such that for all $\mathcal{G}\in \mathcal{U}$,
  \begin{align*}
      \Big|\frac{d^3}{ds^3}\Big|_{s=0} \lambda(\mathcal{G}+s\gamma)  \Big|\leq C\|\gamma\|_{H^1}^2 \|\gamma\|_{C^{2,\alpha}}.
  \end{align*}
\end{lemma}
\begin{proof}
Let $()'$ denote the derivative with respect to variable $s$. We compute 
\begin{align*}
    \frac{d^3}{ds^3}\Big|_{s=0} \lambda(\mathcal{G}+t\gamma)&=   \frac{d^2}{ds^2}\Big|_{s=0} \int_M \langle \Rc^{H,f},\gamma \rangle e^{-f}dV_g
    \\&=\frac{d}{ds}\Big|_{s=0}\Big[\int_M\langle (\Rc^{H,f})',\gamma \rangle e^{-f}dV_g +\int_M(g^{-1})'\circ \Rc^{H,f}\circ \gamma e^{-f}dV_g +\int_M\langle \Rc^{H,f},\gamma \rangle (e^{-f}dV_g)' \Big]
    \\&=\int_M\langle (\Rc^{H,f})'',\gamma \rangle e^{-f}dV_g+2\int_M(g^{-1})'\circ (\Rc^{H,f})'\circ \gamma e^{-f}dV_g +2\int_M(g^{-1})'\circ \Rc^{H,f}\circ \gamma (e^{-f}dV_g)'
    \\&\kern2em +\int_M(g^{-1})'\circ(g^{-1})'\circ \Rc^{H,f}\circ \gamma e^{-f}dV_g+ 2\int_M\langle (\Rc^{H,f})',\gamma \rangle (e^{-f}dV_g)' +\int_M\langle \Rc^{H,f},\gamma \rangle (e^{-f}dV_g)''. 
\end{align*}
Then, we get
\begin{align*}
    (\Rc^{H,f})'&=-\frac{1}{2}\overline{\Delta}_f\gamma-\mathring{R}^+(\gamma)-\frac{1}{2}\overline{\divg}_f^*\overline{\divg}_f\gamma-\frac{1}{2}(\nabla^+)^2(\tr_g\gamma-2f')
    \\(\Rc^{H,f})''&=-\frac{1}{2}(\overline{\Delta}_f)'\gamma-\big(\mathring{R}^+(\gamma)\big)'-\frac{1}{2}\big(\overline{\divg}_f^*\overline{\divg}_f\gamma\big)'-\frac{1}{2}\Big((\nabla^+)^2(\tr_g\gamma-2f')\Big)'
\end{align*}
By (\ref{Est1}) and (\ref{Est2}), we get 
\begin{align}\label{Est3}
    \nonumber (\Rc^{H,f})''&=\gamma\ast\nabla^2\gamma+ \nabla\gamma \ast\nabla\gamma+\nabla\gamma\ast \nabla f'+\nabla\gamma\ast f'\ast\nabla f'+\nabla f'\ast \nabla f'+\nabla^2\gamma\ast\gamma+R^+\ast\gamma\ast\gamma
    \\&\kern2em+\gamma\ast\nabla\gamma\ast H+\nabla^2 f''.
\end{align}
Besides, 
\begin{align*}
    \big(e^{-f}dV_g\big)'&=\big(\frac{1}{2}\tr_g\gamma-f'\big) e^{-f}dV_g
    \\ \big(e^{-f}dV_g\big)''&=\big(\frac{1}{2}\tr_g\gamma-f'\big)^2 e^{-f}dV_g+\big(\frac{1}{2}\tr_g\gamma-f'\big)' e^{-f}dV_g.
\end{align*}
By \Cref{f2}, we see that each term $(\Rc^{H,f})'$, $(\Rc^{H,f})''$, $\int_M (e^{-f}dV_g )'$ and $\int_M (e^{-f}dV_g )''$ is controlled and then
 \begin{align*}
      \Big|\frac{d^3}{dt^3}\Big|_{t=0} \lambda(\mathcal{G}+t\gamma)  \Big|\leq C\|\gamma\|_{H^1}^2 \|\gamma\|_{C^{2,\alpha}}.
  \end{align*}

\end{proof}

\subsection{Local Maximum property}

In the following subsections, we fix a compact manifold $M$, an exact Courant algebroid $E\cong (TM\oplus T^*M)_{H_0}$ with a background closed 3-form $H_0$ and a generalized metric $\mathcal{G}_0$. We fix the following notations.
\begin{itemize}
    \item $S_{\mathcal{G}_0}^f$ is the generalized Ebin's slice constructed in \Cref{slice}.
    \item $ \widetilde{S}_{\mathcal{G}_0}^f=\{\mathcal{G}+\gamma \kern0.2em|  \kern0.2em\overline{\divg}_f\gamma=0 \}$ is the affine generalized slice.
    \item $\mathcal{P}_{\mathcal{G}_0}=\{\mathcal{G}\in S_{\mathcal{G}_0}^f: \Rc^{H,f}=0 \}$ is the premoduli space with respect to the generalized Ebin's slice.
    \item $\widetilde{\mathcal{P}}_{\mathcal{G}_0}=\{\mathcal{G}\in \widetilde{S}_{\mathcal{G}_0}^f: \Rc^{H,f}=0 \}$ is the premoduli space with respect to the affine generalized slice.
\end{itemize}

 In general, the tangent space of the premoduli space may not be the set of infinitesimal generalized solitonic deformation. Instead, we get the following. 

\begin{theorem}[\cite{KK}, Theorem 4.9]
    Let $\mathcal{G}(g,b)$ be a steady gradient generalized Ricci soliton. There exists a neighborhood $\mathcal{U}$ of $\mathcal{G}(g,b)$ in the slice $\mathcal{S}^f_{\mathcal{G}}$ and a 
    finite-dimensional real analytic submanifold $\mathcal{Z}\subset \mathcal{U}$ such that 
    \begin{itemize}
        \item $T_{\mathcal{G}}(\mathcal{Z})=IGSD$
        \item $\mathcal{Z}$ contains the premoduli space $\mathcal{P}_{\mathcal{G}}$ as a real analytic subset.
    \end{itemize}
\end{theorem}
In this subsection, we suppose that all infinitesimal solitonic deformation are integrable so $\mathcal{P}_{\mathcal{G}_0}=\mathcal{Z}$ is a finite dimensional submanifold. Besides, we have the following properties. 
\begin{enumerate}
    \item $\widetilde{\mathcal{P}}_{\mathcal{G}_0}$ is a finite dimensional manifold:  Define a map $\Psi:\mathcal{P}_{\mathcal{G}_0}\longrightarrow \widetilde{S}_{\mathcal{G}_0}^f$ by sending $\mathcal{G}_1 \in \mathcal{P}_{\mathcal{G}_0}$ to the unique generalized metric $\widetilde{\mathcal{G}}_1\in \widetilde{S}_{\mathcal{G}_0}^f$ isometric to $\mathcal{G}_1$. By shrinking $C^{2,\alpha}-$neighborhood and \Cref{affineslice}, we see that $\Psi$ is diffeomorphism onto its image and the set $\widetilde{\mathcal{P}}_{\mathcal{G}_0}=\Psi(\mathcal{P}_{\mathcal{G}_0})$ consisting of all generalized solitons in $\widetilde{S}_{\mathcal{G}_0}^f$ is a finite dimensional manifold.
    \item  Let $\mathcal{W}$ denote the $f-$twisted $L^2$orthogonal complement of $T_{\mathcal{G}_0}\widetilde{\mathcal{P}}_{\mathcal{G}_0}$ in $\ker\overline{\divg}_f$. For any $\mathcal{G}\in\widetilde{S}_{\mathcal{G}_0}^f$, we can find $\widetilde{\mathcal{G}}\in\widetilde{\mathcal{P}}_{\mathcal{G}_0}$ and $\gamma\in\mathcal{W}$ such that $\mathcal{G}=\widetilde{\mathcal{G}}+\gamma$ by the inverse function theorem.
\end{enumerate}

Now, we are ready to prove our main theorem.
\begin{theorem}\label{LM}
   Let $M$ be a compact manifold and $E\cong (TM\oplus T^*M)_{H_0}$ be an exact Courant algebroid with a background closed 3-form $H_0$. Suppose that $\mathcal{G}_0$ is a steady gradient generalized Ricci soliton and all infinitesimal generalized solitonic deformation are integrable. If $\mathcal{G}_0$ is linearly stable, then there exists a $C^{2,\alpha}-$neighborhood $\mathcal{U}$ of $\mathcal{G}_0$ such that for all $\mathcal{G}\in\mathcal{U}$,
   \begin{align*}
       \lambda(\mathcal{G})\leq \lambda(\mathcal{G}_0).
   \end{align*}
   The equality holds if and only if $\mathcal{G}$ is a steady gradient generalized Ricci soliton. 
\end{theorem}
\begin{proof}

As discussed above, for any $\mathcal{G}\in\widetilde{S}_{\mathcal{G}_0}^f$, we can find $\widetilde{\mathcal{G}}\in\widetilde{\mathcal{P}}_{\mathcal{G}_0}$ and $\gamma\in\mathcal{W}$ such that $\mathcal{G}=\widetilde{\mathcal{G}}+\gamma$. Apply the Taylor's expansion, 
\begin{align*}
    \lambda(\mathcal{G})=\lambda(\widetilde{\mathcal{G}})+\Big(\frac{d}{dt}\big|_{t=0} \lambda(\widetilde{\mathcal{G}}+t\gamma) \Big)+\Big(\frac{1}{2}\frac{d^2}{dt^2}\big|_{t=0} \lambda(\widetilde{\mathcal{G}}+t\gamma) \Big)+R(\widetilde{\mathcal{G}},\gamma),
\end{align*}
where the remainder is given by
\begin{align*}
    R(\widetilde{\mathcal{G}},\gamma)=\int_0^1\frac{1}{2}(t-1)^2 \frac{d^3}{dt^3}\big|_{t=0} \lambda(\widetilde{\mathcal{G}}+t\gamma) dt.
\end{align*}

By the third order estimate \Cref{third}, we see that 
\begin{align*}
    \Big| \frac{d^3}{dt^3}\big|_{t=0} \lambda(\widetilde{\mathcal{G}}+t\gamma) \Big|\leq C\|\gamma\|_{C^{2,\alpha}}\|\gamma\|^2_{H^1},
\end{align*}
Next, we note that $\mathcal{G}_0$ is linearly stable so 
\begin{align*}
     \frac{d^2}{dt^2}\big|_{t=0} \lambda(\mathcal{G}_0+t\gamma) \leq -C_\|\gamma\|^2_{H^1}.
\end{align*}
To estimate the difference between $ \frac{d^2}{dt^2}\big|_{t=0} \lambda(\mathcal{G}_0+t\gamma)$ and $ \frac{d^2}{dt^2}\big|_{t=0} \lambda(\widetilde{\mathcal{G}}+t\gamma)$, we write $\widetilde{\mathcal{G}}=\mathcal{G}_0+\delta$. Similar to the argument in third order estimate, the second order operator $N_f$ is bounded. More precisely,
\begin{align*}
    \Big|  \frac{d}{ds}\big|_{s=0} (N_f)_{\mathcal{G}_0+s\delta}(\gamma)  \Big| \leq C\|\delta\|_{C^{2,\alpha}}\|\gamma\|_{H^1}.
\end{align*}
Then, we get 
\begin{align*}
    \Big|\frac{d^2}{dt^2}\big|_{t=0} \lambda(\widetilde{\mathcal{G}}+t\gamma) -\frac{d^2}{dt^2}\big|_{t=0} \lambda(\mathcal{G}_0+t\gamma) \Big|\leq C\|\delta\|_{C^{2,\alpha}}\|\gamma\|_{H^1}.
\end{align*}
By choosing $\delta$ small enough, we have
\begin{align*}
     \frac{d^2}{dt^2}\big|_{t=0} \lambda(\widetilde{\mathcal{G}}+t\gamma)\leq -C\|\gamma\|^2_{H^1}.
\end{align*}
By shrinking the neighborhood $\mathcal{U}$, we have $ \lambda(\mathcal{G})\leq \lambda(\widetilde{\mathcal{G}})=\lambda(\mathcal{G}_0).$ Due to \Cref{affineslice}, any $\mathcal{G}\in\mathcal{U}$ is isomorphic to some generalized metric in the affine slice. Since $\lambda$ is diffeomorphism invariant, the inequality holds for any $\mathcal{G}\in\mathcal{U}$. Moreover, the equality holds if and only if $\mathcal{G}=(\varphi_X,B)\cdot \overline{\mathcal{G}}$ for some $(\varphi_X,B)\in \GDiff^e_H$ and $\overline{\mathcal{G}}\in\mathcal{P}_{\mathcal{G}_0}$.

\end{proof}

\subsection{Optimal Lojasiewicz--Simon Inequality and Transversality}

\begin{theorem}[Optimal Lojasiewicz--Simon Inequality] \label{OPL}
  Let $M$ be a compact manifold and $E\cong (TM\oplus T^*M)_{H_0}$ be an exact Courant algebroid with a background closed 3-form $H_0$. Suppose that $\mathcal{G}_0$ is a steady gradient generalized Ricci soliton and all infinitesimal generalized solitonic deformation are integrable. There exists a $C^{2,\alpha}-$neighborhood $\mathcal{U}$ of $\mathcal{G}_0$ and a constant $C$ such that 
  \begin{align*}
      |\lambda(\mathcal{G})-\lambda(\mathcal{G}_0)|^{\frac{1}{2}}\leq C\|\Rc^{H,f}\|_{L^2}.
  \end{align*}
\end{theorem}

Note that in the proof of the convergence of the GRF, the optimal Lojasiewicz--Simon inequality will ensure that the GRF converges exponentially as $t\to\infty$. In addition, if we assume that $\mathcal{G}_0$ is a generalized Einstein metric. we could further derive the transversality which ensures that the GRF does not move too excessively in gauge direction.

\begin{theorem}[Transversality] \label{TRA}
 Let $M$ be a compact manifold and $E\cong (TM\oplus T^*M)_{H_0}$ be an exact Courant algebroid with a background closed 3-form $H_0$. Suppose that $\mathcal{G}_0$ is a generalized Einstein metric and all infinitesimal generalized solitonic deformations are integrable. There exists a $C^{2,\alpha}-$neighborhood $\mathcal{U}$ of $\mathcal{G}_0$ and a constant $C$ such that 
 \begin{align*}
     \|\Rc^{H}\|_{L^2}\leq C \|\Rc^{H,f}\|_{L^2}
 \end{align*}
 for all $\mathcal{G}(g,b)\in \mathcal{U}$ and $f$ is the minimizer of $\lambda(g,b)$. 
\end{theorem}

 With these two properties, we will conclude that the GRF converges in the strict sense without pulling back by a family of diffeomorphisms in the generalized Einstein case. In this subsection, we begin to prove these two properties. Due to diffeomorphism invariance, it suffices to prove both theorems on a slice. As we discussed in this section, we could work on the affine generalized slice $\widetilde{S}_{\mathcal{G}_0}$. Adopt the same notation in the subsection 4.3, every generalized metric $\mathcal{G}\in \widetilde{S}_{\mathcal{G}_0}$ can be uniquely written as 
\begin{align*}
    \mathcal{G}=\widetilde{\mathcal{G}}+\gamma,
\end{align*}
where $\widetilde{\mathcal{G}}\in \widetilde{\mathcal{P}}_{\mathcal{G}_0}$ and $\gamma\in\mathcal{W}$, provided that $\mathcal{U}$ is small enough. To obtain some estimates, we first need some lemmas.

\begin{lemma}
 Let $\mathcal{G}_0$ be a steady gradient generalized Ricci soliton and write $ \mathcal{G}=\widetilde{\mathcal{G}}+\gamma\in \widetilde{S}_{\mathcal{G}_0}$ as above, we have the expansion
 \begin{align*}
     \Rc^{H,f}_{\mathcal{G}}=L_f(\gamma)+O_1+O_2,
 \end{align*}
 where
 \begin{align*}
     &O_1=\int_0^1 (1-t)\frac{d^2}{dt^2}\Rc^{H,f}_{\mathcal{G}_0+t\gamma} dt,
     \\&O_2=\int_0^1\int_0^1 \frac{\partial^2}{\partial s \partial t} \Rc^{H,f}_{\mathcal{G}_0+s(\widetilde{\mathcal{G}}-\mathcal{G}_0)+t\gamma}dsdt.
 \end{align*}
\end{lemma}

\begin{proof}
 We recall  Lemma 5.4.7 in \cite{Kr}. Let $F(s,t)$ be a $C^2-$function on $0\leq s,t\leq 1$ with values in a Fréchet space. Then, 
\begin{align*}
    F(1,1)=F(1,0)+\frac{d}{dt}\Big|_{t=0}F(0,t)+\int_0^1 (1-t)\frac{d^2}{dt^2}F(0,t)dt+\int_0^1\int_0^1 \frac{\partial^2}{\partial s \partial t}F(s,t)dsdt.
\end{align*}
The proof is complete if we write
\begin{align*}
    F(s,t)= \Rc^{H,f}_{\mathcal{G}_0+s(\widetilde{\mathcal{G}}-\mathcal{G}_0)+t\gamma}.
\end{align*}

\end{proof}

\begin{lemma} \label{Oest}
  Let $\mathcal{G}_0$ be a steady gradient generalized Ricci soliton and write $ \mathcal{G}=\widetilde{\mathcal{G}}+\gamma\in \widetilde{S}_{\mathcal{G}_0}$ as above. Then, there exists a $C^{2,\alpha}-$neighborhood and a constant $C$ such that in the neighborhood
  \begin{align*}
      \|O_1\|_{L^2}\leq C\|\gamma\|_{C^{2,\alpha}}\|\gamma\|_{H^2}, \quad 
      \|O_2\|_{L^2}\leq C\|\widetilde{\mathcal{G}}-\mathcal{G}_0\|_{C^{2,\alpha}}\|\gamma\|_{H^2}.
  \end{align*}
\end{lemma}

\begin{proof}
The proof is directly followed by the second derivative estimate of $\Rc^{H,f}$ (\ref{Est3}).

\end{proof}

Now we can discuss the proof of \Cref{OPL} and \Cref{TRA}.
\begin{proof}
  Using the Taylor's expansion, we have 
  \begin{align*}
      |\lambda(\mathcal{G})-\lambda(\mathcal{G}_0)|=|\lambda(\widetilde{\mathcal{G}}+\gamma)-\lambda(\widetilde{\mathcal{G}})|\leq C\|\gamma\|^2_{H^2}.
  \end{align*}
Then, \Cref{Oest} shows that
\begin{align*}
    \| \Rc^{H,f}_{\mathcal{G}}\|_{L^2}\geq \|L_f(\gamma)\|_{L^2}-\big( \|O_1\|_{L^2}+ \|O_2\|_{L^2} \big).
\end{align*}
Note that $L_f$ is injective in space $\mathcal{W}$, so $\|L_f(\gamma)\|_{L^2}\geq C\|\gamma\|_{H^2}$. We can shrink $\mathcal{U}$ small enough such that 
\begin{align} \label{Est4}
    \| \Rc^{H,f}_{\mathcal{G}}\|_{L^2}\geq C\|\gamma\|_{H^2}.
\end{align}
Then, the proof of \Cref{OPL} is done. If we assume $\mathcal{G}_0$ is a generalized Einstein metric, then use the Taylor's expansion again, we have
\begin{align*}
    \| \Rc^{H}_{\mathcal{G}}-\Rc^{H}_{\mathcal{G}_0} \|_{L^2}=\| \Rc^{H}_{\mathcal{G}} \|_{L^2}\leq C \|\gamma\|^2_{H^2}.
\end{align*}
The proof of \Cref{TRA} is followed by (\ref{Est4}).
\end{proof}

\subsection{Dynamical stability}

As a corollary, we can improve \cite{K} Theorem 1.8.
\begin{theorem} \label{MTT}
    
     Let $M$ be a compact manifold and $E\cong (TM\oplus T^*M)_{H_c}$ be an exact Courant algebroid with a background closed 3-form $H_c$. Suppose that $\mathcal{G}_c$ is a steady gradient generalized Ricci soliton and all infinitesimal generalized solitonic deformations are integrable. If $\mathcal{G}_c$ is linearly stable and $k\geq 2$, then for every $C^k$-neighborhood $\mathcal{U}$ of $\mathcal{G}_c$ , there exists some $C^{k+2}$-neighborhood $\mathcal{V}$ such that the following holds:

     For any GRF $\mathcal{G}_t$ starting at $\mathcal{G}_0\in \mathcal{V}$, there exists a family of automorphism $\{(\varphi_t,B_t)\}\in \GDiff_{H_c}$ such that the modified flow $(\varphi_t,B_t)\cdot\mathcal{G}_t$ stays in $\mathcal{U}$ for all time and converges exponentially to a steady gradient generalized Ricci soliton$\mathcal{G}_\infty$. More precisely, there exists constants $C_1,C_2$ such that for all $t\geq 0$, 
     \begin{align*}
         \|(\varphi_t,B_t)\cdot\mathcal{G}_t-\mathcal{G}_\infty\|_{C^k_{\mathcal{G}_c}}\leq C_1 e^{-C_2t}.
     \end{align*}
     
     If $\mathcal{G}_c$ is a generalized Einstein metric, any GRF $\mathcal{G}_t$ starting at $\mathcal{G}_0\in \mathcal{V}$ stays in $\mathcal{U}$ for all time and converges exponentially to a steady gradient generalized generalized Ricci soliton $\mathcal{G}_\infty$.
\end{theorem}

\begin{proof}
    Most of the proof is similar to \cite{K} Theorem 1.8. To be self-contained, let us briefly discuss the proof here. Apply the Lemma 5.10 in \cite{K} and the optimal Lojasiewicz--Simon inequality (\Cref{OPL}), we can take $\epsilon$ small enough such that in the $C^k_{\mathcal{G}_c}-$neighborhood $\mathcal{U}=B^k_\epsilon$ of $\mathcal{G}_c$ such that 
    \begin{itemize}
        \item For any $\mathcal{G}$ with $ \|\mathcal{G}-\mathcal{G}_c\|_{C^{k}_{\mathcal{G}_c}}<\epsilon$, we have
        \begin{align} \label{App}
            |\lambda(\mathcal{G})-\lambda(\mathcal{G}_c)|^{\frac{1}{2}}\leq C\| \Rc^{H,f}\|_{L^2}.
        \end{align}
        \item We can choose a smaller $C^{k+2}_{\mathcal{G}_c}$-neighborhood $\mathcal{V}$ such that the GRF starting in $\mathcal{V}$ will stay in $B^k_{\frac{\epsilon}{4}}$ up to time 1.
    \end{itemize}
    
 Assume $T\geq 1$ be the maximal time such that for any solution of GRF starting in $\mathcal{V}$, there exists a family of automorphism $\{(\varphi_t,B_t)\}$ such that 
\begin{align*}
    \|(\varphi_t,B_t)\cdot\mathcal{G}_t-\mathcal{G}_c\|_{C^{k}_{\mathcal{G}_c}}<\epsilon \quad \text{ for all $t\leq T$.}
\end{align*}   
Let $\mathcal{G}_t$ be a solution of the generalized Ricci flow starting with $\mathcal{G}_0\in\mathcal{V}$. We define the modified flow $\widetilde{\mathcal{G}}_t$ starting with $\mathcal{G}_0$ as follows.

\begin{itemize}
    \item For $t\leq 1$: Define $\widetilde{\mathcal{G}}_t=\mathcal{G}_t$.
    \item For $t\geq 1$: We define $\widetilde{\mathcal{G}}_t(\widetilde{g}_t,\widetilde{b}_t)$ by the solution of $(-\nabla_{g_t}f_{\mathcal{G}_t},0)$-gauge fixed generalized Ricci flow. More precisely, 
\begin{align*}
    &\frac{\partial \widetilde{g}}{\partial t}=-2(\widetilde{\Rc}-\frac{1}{4}\widetilde{H}^2+\nabla^2\widetilde{f}), \quad \widetilde{g}(1)=g(1),
    \\&\frac{\partial \widetilde{b}}{\partial t}=-(d^*\widetilde{H}+i_{\nabla\widetilde{f}}\widetilde{H}), \quad \widetilde{b}(1)=b(1). 
\end{align*}
\end{itemize}

Denote $\psi_t$ the diffeomorphism generated by $X(t)=-\nabla_{g_t}f_{\mathcal{G}_t}$ and $\psi_1=Id$, the modified flow $\widetilde{\mathcal{G}}_t$ can also be expressed as
\begin{align}
   \widetilde{\mathcal{G}}_t=\begin{cases}\mathcal{G}(t) & t\in[0,1] \\ (\psi_t, B_t) \cdot \mathcal{G}(t) & t\geq 1
    \end{cases} ,
\end{align}
where $B_t=\iota_{\nabla_{g_t}f_{\mathcal{G}_t}}H_0-d(\iota_{\nabla_{g_t}f_{\mathcal{G}_t}}b_t)$.

Let $\overline{T}\geq 1$ be the maximal time such that the modified flow $\widetilde{\mathcal{G}}_t$ satisfies
\begin{align*}
    \|\widetilde{\mathcal{G}}_t-\mathcal{G}_c\|_{C^{k}_{\mathcal{G}_c}}<\epsilon, \quad \text{ for all $t\leq \overline{T}$}. 
\end{align*}
Using the Hamilton's interpolation theorem (\cite{10.4310/jdg/1214436922} Corollary 12.7) and Lemma 5.11 in \cite{K}, we find a $\beta\in(0,1)$ such that 
\begin{align*}
    \|(\frac{\partial \widetilde{g}}{\partial t},\frac{\partial \widetilde{b}}{\partial t})\|_{C^{k}_{\widetilde{g}(t)}}
    \leq   C\|(\frac{\partial \widetilde{g}}{\partial t},\frac{\partial \widetilde{b}}{\partial t})\|^{\beta}_{L^2_{\widetilde{g}(t)}}. 
\end{align*}
Note that $\lambda(\mathcal{G}_c)$ is a local maximum by \Cref{LM}, we use (\ref{App}) and compute
\begin{align}\label{Kinequality}
   \nonumber \kern-1em-\frac{d}{dt}|\lambda(\widetilde{\mathcal{G}_t})-\lambda(\mathcal{G}_c)|^{\frac{\beta}{2}}&=\frac{\beta}{2} |\lambda(\widetilde{\mathcal{G}_t})-\lambda(\mathcal{G}_c)|^{\frac{\beta}{2}-1}\frac{d}{dt}\lambda(\widetilde{\mathcal{G}_t})
    \nonumber\\&= \frac{\beta}{2}| \lambda(\widetilde{\mathcal{G}_t})-\lambda(\mathcal{G}_c)|^{\frac{\beta}{2}-1}\int_M2|\widetilde{\Rc}^{H,f}|^2 e^{-\widetilde{f}}dV_{\widetilde{g}}
    \nonumber\\&\geq C \|\widetilde{\Rc}^{H,f} \|^\beta_{L^2_{\widetilde{g}(t)}} 
   \nonumber \\&\geq C \|\widetilde{\Rc}^{H,f} \|^\beta_{C^{k}_{\widetilde{g}(t)}}.
\end{align}
Then, 
\begin{align*}
    \int_1^{\overline{T}}\|(\frac{\partial \widetilde{g}}{\partial t},\frac{\partial \widetilde{b}}{\partial t})\|_{C^{k}_{\widetilde{g}(t)}} dt&\leq C \int_1^{\overline{T}}(-\frac{d}{dt}|\lambda(\widetilde{\mathcal{G}_t})-\lambda(\mathcal{G}_c)|^{\frac{\beta}{2}}) dt
    \\&= C\left(|\lambda(\widetilde{\mathcal{G}}(1))-\lambda(\mathcal{G}_c)|^{\frac{\beta}{2}}-|\lambda(\widetilde{\mathcal{G}}(\overline{T}))-\lambda(\mathcal{G}_c)|^{\frac{\beta}{2}}\right)
    \\&\leq C|\lambda(\mathcal{G}(0))-\lambda(\mathcal{G}_c)|^{\frac{\beta}{2}} <\frac{\epsilon}{4},
\end{align*}
provided when $\mathcal{V}$ is small enough. By shrinking $\epsilon$ small enough, we may suppose that the $C^k$-norms with respect to $\widetilde{g}(t)$ and $g_c$ differ at most by a factor 2. Then,
\begin{align*}
    \|\widetilde{\mathcal{G}}(\overline{T})-\mathcal{G}_c\|_{C^{k}_{g_c}}&\leq  \|\widetilde{\mathcal{G}}(1)-\mathcal{G}_c\|_{C^{k}_{g_c}}+ \int_1^{\overline{T}}\frac{d}{dt} |\widetilde{\mathcal{G}}(t)-\mathcal{G}_c|_{C^{k}_{g_c}}dt
    \\&\leq \frac{\epsilon}{4}+2 \int_1^{\overline{T}}\|(\frac{\partial \widetilde{g}}{\partial t},\frac{\partial \widetilde{b}}{\partial t})\|_{C^{k}_{\widetilde{g}(t)}} dt\leq \frac{3\epsilon}{4}
\end{align*}
which is a contradiction. Therefore, $\overline{T}=T=\infty$ and $\widetilde{\mathcal{G}}_t\longrightarrow \mathcal{G}_\infty$. Moreover, we see that 
\begin{align*}
    -\frac{d}{dt}|\lambda(\widetilde{\mathcal{G}_t})-\lambda(\mathcal{G}_c)|=\frac{d}{dt}\lambda(\widetilde{\mathcal{G}_t})=2\|\widetilde{\Rc}^{H,f}\|_{L^2}\geq C |\lambda(\widetilde{\mathcal{G}_t})-\lambda(\mathcal{G}_c)|.
\end{align*}
Hence, $\lambda(\mathcal{G}_\infty)=\lambda(\mathcal{G}_c)$ and the convergence is exponential. Moreover, $\lambda(\mathcal{G}_\infty)$ is also a local maximum so $\mathcal{G}_\infty$ must be a compact steady gradient generalized Ricci soliton.

In the generalized Einstein case, it is unnecessary to consider the modified flow. In the argument of the inequality (\ref{Kinequality}), by transversality \Cref{TRA} we get
\begin{align*}
    -\frac{d}{dt}|\lambda(\mathcal{G}_t)-\lambda(\mathcal{G}_c)|^{\frac{\beta}{2}}\geq C \|\Rc^{H,f} \|^\beta_{C^{k}_{g(t)}}.
\end{align*}
The rest of argument are all the same.
\end{proof}

\begin{remark} \label{Autconv}
 In the proof of \Cref{MTT}, we actually argue that the $(-\nabla_{g_t}f_{\mathcal{G}_t},0)$-gauge fixed generalized Ricci flow exists for long time and converges to some steady gradient generalized Ricci soliton.     
\end{remark}

\begin{remark}
  In \cite{KK} Theorem 1.7, the author computed the second-order integrability of $S^3$ and  found that some of the infinitsimal solitonic deformations are not integrable. Hence, we cannot guarantee that $\lambda(g_{S^3},H_{S^3})$ is a local maximum and $(g_{S^3},H_{S^3})$ is dynamically stable.   
\end{remark}

\section{Infinitesimal deformations in the pluriclosed setting}

\subsection{Second variation in the pluriclosed setting}
Starting in this section, we focus mainly on the complex geometry. Let $(M,g,J)$ be a Hermitian manifold and $\omega\in \Lambda^{1,1}_{\mathbb{R}}$ be the associated Kähler form $\omega\in \Lambda^{1,1}_{\mathbb{R}}$ given by
\begin{align*}
    \omega(X,Y)=g(JX,Y), \quad \text{for } X,Y\in TM.
\end{align*}
We say that a Hermitian metric $g$ on $M$ is pluriclosed if 
\begin{align*}
    dd^c\omega=0,
\end{align*}
where $d^c=\sqrt{-1}(\overline{\partial}-\partial)$. The associated Lee form $\theta$ is defined by 
\begin{align*}
    \theta=-d^*\omega\circ J. 
\end{align*}

In terms of the Hermitian manifold, let us formally define the Bismut connection.
\begin{defn} \label{DDD}
Let $(M,g,J)$ be a Hermitian manifold.

\begin{itemize}
    \item The \emph{Bismut connection} $\nabla^B$ is defined by    
\begin{align}
    \langle \nabla_X^{B}Y,Z \rangle=\langle \nabla_XY,Z \rangle- \frac{1}{2}d^c\omega(X,Y,Z). \label{20}
\end{align}
In fact, $\nabla^B$ is a Hermitian connection and equals $\nabla^+$ with $H=-d^c\omega$.
\item The \emph{Bismut Ricci curvature} $\rho_{B}$ of $g$ is given by 
\begin{align}
    \rho_{B}(X,Y)\coloneqq \frac{1}{2}\langle R^{B}(X,Y)Je_i,e_i\rangle, \label{BRF}
\end{align}
where $R^{B}$ denotes the Riemann curvature tensor with respect to the connection $\nabla^{B}$ and $\{e_i\}$ is any orthonormal basis for $T_pM$ at any given point $p$. We say that a pluriclosed metric $g$ is a \emph{Bismut-Hermitian-Einstein metric} if $\rho_B=0$. 
\item We say that a Hermitian manifold $(M,g,H,J,f)$ is called a \emph{pluriclosed steady soliton} if $g$ is a pluriclosed metric, $H=-d^c\omega=H_0+db$ and $(M,g,H,J,f)$ satisfies (\ref{s}).
\item The \emph{Laplace operator of the Bismut connection} $\Delta^B_f$ is defined by 
\begin{align*}
    \Delta^B_f=-(\nabla^B)^{*_f}\nabla^B, 
\end{align*}
where $(\nabla^B)^{*_f}$ is the formal adjoint of $\nabla^B$ with respect to $f$-twisted $L^2$ inner product (\ref{6}). In real coordinates, we get
\begin{align}
    \Delta^B_f\gamma_{ij}=\Delta_f\gamma_{ij}-H_{mjk}\nabla_m\gamma_{ik}-H_{mik}\nabla_m\gamma_{kj}-\frac{1}{4}(H_{jl}^2\gamma_{il}+H_{il}^2\gamma_{lj})+\frac{1}{2}H_{mkj}H_{mli}\gamma_{lk}. \label{LB}
\end{align}
\end{itemize}

\end{defn}

\begin{example}
Given complex numbers $\alpha,\beta$ satisfying $0<|\alpha|\leq |\beta|<1$, we obtain a Hopf surface 
\begin{align*}
    \mathbb{C}^2\setminus\{0\}/\langle(z_1,z_2)\rightarrow (\alpha z_1,\beta z_2) \rangle.
\end{align*}
For any $\alpha,\beta$ this is a complex manifold diffeomorphic to $S^3\times S^1$. In particular, we say that the Hopf surface is diagonal if $|\alpha|=|\beta|$. For diagonal Hopf surfaces, consider the Hopf/Boothby metric defined by the invariant Kähler form on $ \mathbb{C}^2\setminus\{0\}$ by 
\begin{align*}
    \omega_{\text{Hopf}}=\frac{\sqrt{-1}}{|z|^2}(dz_1\wedge d\overline{z}_1+dz_2\wedge d\overline{z}_2).
\end{align*}
In \cite{GRF} Proposition 8.25, one deduces that this metric is a pluriclosed metric, In fact, we observe that the metric on the universal cover $ \mathbb{C}^2\setminus\{0\}$ is isometric to the standard cylinder $(S^3\times\mathbb{R}, g_{S^3}\oplus dt^2)$. Therefore, this metric is Bismut-flat and $\rho_B=0$. In \cite{J5} Theorem 1.1, Streets constructed a non-trivial  pluriclosed steady
soliton on Hopf surface. Moreover, detailed discussion about non-trivial pluriclosed steady solitons on complex surface could be found in \cite{J7} and \cite{J4}.     
\end{example}

In the following, let $(M,g,H,J,f)$ be a pluriclosed steady soliton with a closed 3-form $H=-d^c\omega$. Consider a family $(M,g_t,H_t,J_t,f_t)$ with 
\begin{align*}
    H_t=H_0+db_t,\quad (M,g_0,H_0,J_0,b_0)=(M,g,H,J,0)
\end{align*}
and $f_t$ is the minimizer of $\lambda(g_t,b_t)$. Denote 
\begin{align*}
   \frac{\partial }{\partial t}\big|_{t=0} g=h, \quad  \frac{\partial }{\partial t}\big|_{t=0} b=\beta,\quad   \frac{\partial }{\partial t}\big|_{t=0}\omega =\phi, \quad  \frac{\partial }{\partial t}\big|_{t=0} J =I,
\end{align*}
where $I$ is an essential infinitesimal variation of complex structure. Following the discussion in section 3, we may take our general variation  $\gamma=h-\beta\in\ker\overline{\divg}_f$ (See (\ref{bar1}) for the definition of $\overline{\divg}_f$). Suppose that $g_t$ is compatible with $J_t$, thus for all vector fields $X,Y$
\begin{align*}
    \frac{\partial }{\partial t}\big|_{t=0} g_t(X,Y)&= \frac{\partial }{\partial t}\big|_{t=0}\big(\omega_t(X,J_tY) \big)
     \\&=\phi(X,JY)+\omega(X,IY). 
\end{align*}
Then,
\begin{align*}
    \gamma(X,Y)&=h(X,Y)-\beta(X,Y)
    \\&=\phi(X,JY)+\omega(X,IY)-\beta(X,Y),
\end{align*}
and
\begin{align*}
    \gamma(X,JY)=-\phi(X,Y)-\beta(X,JY)-g(X,IY).
\end{align*}

In the following, we define $\xi\coloneqq \phi+\beta\circ J$, $\eta\coloneqq\omega\circ I$ and $\widetilde{\eta}\coloneqq\eta\circ J=-g\circ I$ so that 
\begin{align}
    \gamma\circ J=-\xi+\widetilde{\eta}. \label{GammaJ} 
\end{align}

By \cite{KKK} Proposition 3.3, we see that $\xi$ and $\eta$ satisfy the following properties:
\begin{itemize}
    \item $\xi$ is a 2-form with $(d^B_f)^*\xi=0$.
    \item  $\eta$ is a symmetric, anti-Hermitian 2-tensor with $\divg_f^B\eta=0$. Using the complex coordinate $\{e_\alpha\}$ constructed in \cite{KKK} Proposition 2.16, we have 
    \begin{align} 
     (\nabla^B_\alpha\eta)_{\beta\gamma}-(\nabla^B_\beta\eta)_{\alpha\gamma}=H_{\alpha\overline{\tau}\gamma}\eta_{\beta\tau}-H_{\alpha\beta\overline{\tau}}\eta_{\tau\gamma}-H_{\beta\overline{\tau}\gamma}\eta_{\alpha\tau}. \label{etacommute}
 \end{align}
\end{itemize}
Next, we further define $C\in\Omega^3$ by 
\begin{align}
    C_{ijk}=\widetilde{\eta}_{il}H_{ljk}+\widetilde{\eta}_{jl}H_{lki}+\widetilde{\eta}_{kl}H_{lij}. \label{C}
\end{align}
In fact, due to (\ref{etacommute}), we see that $C\in\Lambda^{2,1}+\Lambda^{1,2}$.

\begin{remark}
    If we assume that $(g_t,H_t,J_t,f_t)$ preserves the pluriclosed structure, i.e., $H_t=-d^c\omega_t=H_0+db_t$. Then, for any vector fields $X,Y,Z$,
    \begin{align*}
      \frac{\partial }{\partial t}\big|_{t=0} (d\omega)(X,Y,Z)&=  -\frac{\partial }{\partial t}\big|_{t=0}H(J_tX,J_tY,J_tZ)
      \\&=-d\beta(JX,JY,JZ)-H(IX,JY,JZ)-H(JX,IY,JZ)-H(JX,JY,IZ)
      \\&=(d\phi)(X,Y,Z).
    \end{align*}
On $(2,1)$-part, we get 
\begin{align*}
    &H(I\partial_{\alpha},J\partial_\beta,J\partial_{\overline{\gamma}})+H(J\partial_{\alpha},I\partial_\beta,J\partial_{\overline{\gamma}})+H(J\partial_{\alpha},J\partial_\beta,I\partial_{\overline{\gamma}})
    \\&\kern2em =g(I\partial_\alpha,\partial_{\tau})H_{\overline{\tau}\beta\overline{\gamma}}+g(I\partial_\beta,\partial_{\tau})H_{\alpha\overline{\tau}\overline{\gamma}}
    \\&\kern2em=-\widetilde{\eta}_{\alpha\tau}H_{\overline{\tau}\beta\overline{\gamma}}-\widetilde{\eta}_{\beta\tau}H_{\alpha\overline{\tau}\overline{\gamma}}
    \\&\kern2em= -C_{\alpha\beta\overline{\tau}}.
\end{align*}
Thus, we conclude that 
\begin{align*}
    d\xi=d(\phi+\beta\circ J)=C
\end{align*}
    
\end{remark}

Next, we briefly recall the pluriclosed flow and its relation between GRF. More details can be found in \cite{GRF}
\begin{defn}
   Let $(M^n,J)$ be a complex manifold with a fix closed three form $H_0$. A one parameter family of pluriclosed metrics $(g_t,b_t)$ is a solution of \emph{pluriclosed flow} if the associated Kähler form $\omega_t$ and $(2,0)-$ forms $\beta_t=\sqrt{-1}b^{2,0}$ satisfy
   \begin{align*}
       \frac{\partial}{\partial t}\omega=-\rho_B^{1,1}, \quad \frac{\partial}{\partial t}\beta=-\rho_B^{2,0}.
   \end{align*}
\end{defn}

\begin{proposition}[\cite{GRF}, Proposition 9.8]\label{PGG}
    Let $(M^n,g_t,J)$ be a solution to pluriclosed flow. Then, the associated pair $(g_t,b_t)$ satisfies 
    \begin{align*}
        &\frac{\partial}{\partial t}g=-\Rc+\frac{1}{4}H^2-\frac{1}{2}\mathcal{L}_{\theta^\sharp}g,
        \\&\frac{\partial}{\partial t}b=-\frac{1}{2}d^*H+\frac{1}{2}d\theta-\frac{1}{2}i_{\theta^\sharp}H.
    \end{align*}
    In particular, pluriclosed flow is gauge equivalent to the generalized Ricci flow.
\end{proposition}

At the end of this section, let me recall the second variation formula in the pluriclosed setting

\begin{theorem}\label{T2}

Let $(M,g,H,J,f)$ be a compact pluriclosed steady soliton and $\gamma\in\ker\overline{\divg}_f$. Then, 
\begin{align}
     \frac{d^2}{dt^2}\Big|_{t=0}\lambda(\gamma)&= -2\|d^*_f\xi\|_f^2-\frac{1}{6}\|d\xi-C\|_f^2 +\int_M |\eta|^2\big( -\frac{1}{2}S_B+\mathcal{L}_Vf \big)e^{-f}dV_g, \label{k4}
\end{align} 
where $\gamma=\xi\circ J+\eta$, $C$ are defined in (\ref{GammaJ}) and (\ref{C}), $\|\cdot\|_f$ denotes the norm of $f$-twisted $L^2$ inner product (\ref{6}),  $S_B=\tr_\omega \rho_B$ denotes the Bismut scalar curvature and the vector field $V=\frac{1}{2}(\theta^\sharp-\nabla f)$. 
    
\end{theorem}

\subsection{The case for a fixed complex structure}

In this subsection, we start discussing the $IGSD$. First, we recall that $\gamma\in\ker\overline{\divg}_f$ is an infitesimal generalized solitonic deformation if $\overline{L}_f(\gamma)=0$. 

\begin{lemma} \label{Lequivalent}
    For any two tensor $\gamma$, $\overline{L}_f(\gamma)=0$ if and only if $\overline{L}_f(\gamma\circ J)=0$
\end{lemma}
\begin{proof}
Comparing (\ref{Lbar}) and (\ref{LB}), we get 
\begin{align*}
    \overline{L}_f(\gamma)_{ij}=\frac{1}{2}\overline{\Delta}_f(\gamma)_{ij}+\mathring{R}^B(\gamma)_{ij}=\frac{1}{2}\Delta_f^B\gamma_{ij}+\mathring{R}^B(\gamma)_{ij}-\nabla^B_m\gamma_{kj}H_{mki}-\frac{1}{2}H^2_{ki}\gamma_{kj}.
\end{align*}
Since $\nabla^BJ=0$, we prove the lemma. 
\end{proof}

Following \Cref{Lequivalent}, for any $\gamma\in IGSD$,
\begin{align*}
    \overline{L}_f(\gamma\circ J)=-\Big(\frac{1}{2}\Delta_f^B\xi_{ij}+\mathring{R}^B(\xi)_{ij}-\nabla^B_m\xi_{kj}H_{mki}-\frac{1}{2}H^2_{ki}\xi_{kj}\Big)+\Big(\frac{1}{2}\Delta_f^B\widetilde{\eta}_{ij}+\mathring{R}^B(\widetilde{\eta})_{ij}-\nabla^B_m\widetilde{\eta}_{kj}H_{mki}-\frac{1}{2}H^2_{ki}\widetilde{\eta}_{kj}\Big)=0
\end{align*}
where $\gamma\circ J=-\xi+\widetilde{\eta}$ is written in (\ref{GammaJ}). In the following, we assume that $\widetilde{\eta}=0$ and say that the variations with $\widetilde{\eta}=0$ are deformations with a fixed complex structure.

\begin{lemma}\label{LDD}
Let $(M,g,H,f)$ be a steady pluriclosed soliton. For any two form $\xi$, 
 \begin{align}
     \overline{L}_f(\xi)_{ij}&=\frac{1}{2}(\Delta_{d,f}\xi)_{ij}-\frac{1}{2}H_{ijk}\big((d_f)^*\xi_k-(d^B_f)^*\xi_k\big)\nonumber
     \\&\kern2em+\frac{1}{2}\nabla_i\big((d_f)^*\xi_j-(d^B_f)^*\xi_j\big)+\frac{1}{2}\nabla_j\big((d_f)^*\xi_i-(d^B_f)^*\xi_i\big)-\frac{1}{4}\big((d\xi)_{mkj}H_{mki}+(d\xi)_{mki}H_{mkj} \big) \label{Lfxi}.
 \end{align}
\end{lemma}

\begin{proof}
Recall that in \cite{KKK} Lemma 3.5, we get the following
\begin{align}\label{LD}
    \Delta^B_{d,f}\xi_{ij}=\Delta^B_f\xi_{ij}+\mathring{R}^B(\xi)_{ij}-\mathring{R}^B(\xi)_{ji}-H_{lik}\nabla^B_k\xi_{jl}-H_{ljk}\nabla^B_k\xi_{li}.
\end{align}
where $\Delta^B_{d,f}$ denotes the $f$-twisted Bismut Hodge Laplacian,  Thus,
\begin{align*}
     \overline{L}_f(\xi)&=\frac{1}{2}\Delta_f^B\xi_{ij}+\mathring{R}^B(\xi)_{ij}-\nabla^B_m\xi_{kj}H_{mki}-\frac{1}{2}H^2_{ki}\xi_{kj}
     \\&=\frac{1}{2}\Delta^B_{d,f}\xi_{ij}+\frac{1}{2}\big(\mathring{R}^B(\xi)_{ij}+\mathring{R}^B(\xi)_{ji}\big)-\frac{3}{2}\nabla^B_m\xi_{kj}H_{mki}+\frac{1}{2}\nabla^B_m\xi_{ki}H_{mkj}-\frac{1}{2}H^2_{ki}\xi_{kj}
     \\&=\frac{1}{2}\Delta^B_{d,f}\xi_{ij}+\frac{1}{2}\big(\mathring{R}^B(\xi)_{ij}+\mathring{R}^B(\xi)_{ji}\big)-\frac{3}{2}\nabla_m\xi_{kj}H_{mki}+\frac{1}{2}\nabla_m\xi_{ki}H_{mkj}+\frac{1}{4}H^2_{ki}\xi_{kj}-\frac{1}{4}H^2_{kj}\xi_{ki}+H_{mil}H_{mjk}\xi_{kl}
\end{align*}
Next, we compute that on steady generalized Ricci solitons the Bismut Hodge Laplacian and Hodge Laplacian is related by 
\begin{align*}
    (\Delta^B_{d,f}\xi)_{ij}=(\Delta_{d,f}\xi)_{ij}-2H_{lik}\nabla_l\xi_{kj}-2H_{jlk}\nabla_l\xi_{ki}-\frac{1}{2}H^2_{ik}\xi_{kj}+\frac{1}{2}H^2_{jk}\xi_{ki}-2H_{mil}H_{mjk}\xi_{kl}-\frac{1}{2}H_{ijk}H_{lmk}\xi_{lm}.
\end{align*}
Then, we have
\begin{align*}
    \overline{L}_f(\xi)&=\frac{1}{2}\Delta_{d,f}\xi_{ij}+\frac{1}{2}\big(\mathring{R}^B(\xi)_{ij}+\mathring{R}^B(\xi)_{ji}\big)-\frac{1}{2}\nabla_m\xi_{kj}H_{mki}-\frac{1}{2}\nabla_m\xi_{ki}H_{mkj}-\frac{1}{4}H_{ijk}H_{lmk}\xi_{lm}.
\end{align*}
Note that $\mathring{R}(\xi)$ is skew-symmetric, we deduce that 
\begin{align*}
    \mathring{R}^B(\xi)_{ij}+\mathring{R}^B(\xi)_{ji}=\frac{1}{2}\big( \nabla_iH_{klj}+\nabla_jH_{kli} \big)\xi_{kl}
\end{align*}
and the symmetric part of $\overline{L}_f(\xi)$ is 
\begin{align*}
   \frac{1}{2}&\big(\mathring{R}^B(\xi)_{ij}+\mathring{R}^B(\xi)_{ji}\big)-\frac{1}{2}\nabla_m\xi_{kj}H_{mki}-\frac{1}{2}\nabla_m\xi_{ki}H_{mkj}
   \\&=\frac{1}{4}\big( \nabla_iH_{klj}+\nabla_jH_{kli} \big)\xi_{kl}-\frac{1}{2}\nabla_m\xi_{kj}H_{mki}-\frac{1}{2}\nabla_m\xi_{ki}H_{mkj}
   \\&=\frac{1}{4}\big( \nabla_iH_{klj}+\nabla_jH_{kli} \big)\xi_{kl}-\frac{1}{4}(d\xi)_{mkj}H_{mki}-\frac{1}{4}(d\xi)_{mki}H_{mkj}+\frac{1}{4}\nabla_j\xi_{mk}H_{mki}+\frac{1}{4}\nabla_i\xi_{mk}H_{mkj}
   \\&=\frac{1}{4}\Big(  \nabla_i(H_{mkj}\xi_{mk})+ \nabla_j(H_{mki}\xi_{mk})-(d\xi)_{mkj}H_{mki}-(d\xi)_{mki}H_{mkj} \Big).
\end{align*}
Since $(d^B_f)^*\xi_k=(d_f)^*\xi_k-\frac{1}{2}H_{lmk}\xi_{lm}$, we complete the proof.
\end{proof}

Following \Cref{LDD}, we conclude that  
\begin{proposition}
   Let $(M,g,H,J,f)$ be a compact steady pluriclosed soliton. For any two form $\xi$,
   \begin{align*}
        \overline{L}_f(\xi)=0\Longleftrightarrow d\xi=0,\quad d^*_f\xi=0,\quad (d^B_f)^*\xi=0.
   \end{align*}
  Therefore, if $\gamma=\xi\circ J\in IGSD$, then $\gamma\circ J$ is $f-$twisted harmonic, i.e, $\Delta_{d,f}(\gamma\circ J)=0$.
\end{proposition}
\begin{proof}
    The $f$-twisted $L^2$ inner product of $ \overline{L}_f(\xi)$ and $\xi$ is given by
    \begin{align*}
        \int_M\langle  \overline{L}_f(\xi),\xi \rangle e^{-f}dV_g&=\int_M \big[ \frac{1}{2}\langle \Delta_{d,f}\xi,\xi\rangle-\frac{1}{2}H_{ijk}\xi_{ij}\big((d_f)^*\xi_k-(d^B_f)^*\xi_k\big) \big] e^{-f}dV_g
        \\&=-\frac{1}{6}\| d\xi \|^2_f-\|d^*_f\xi\|^2_f-\|d^*_f\xi-(d^B_f)^*\xi\|^2_f.
    \end{align*}
    Comparing with (\ref{LD}), we complete the proof.
\end{proof}

In the Bismut-flat case, we could further deduce the following.
\begin{corollary}\label{BFIGSD}
    Let $(M,g,H,J)$ be a compact Bismut-flat manifold. For any two form $\xi$,
    \begin{align*}
         \overline{L}_f(\xi)=0\Longleftrightarrow \xi \text{ is $\nabla^{\pm}$-parallel.}
    \end{align*}
  Therefore, if $\gamma=\xi \circ J\in IGSD$ then $\gamma$ is $\nabla$-parallel and $\nabla^\pm$-parallel.
\end{corollary}
\begin{proof}
    Using the integration by part, we compute that 
    \begin{align*}
        \int_M |d\xi|^2 dV_g&=\int_M 3|\nabla\xi|^2+6 \nabla_l\xi_{ij}\nabla_i\xi_{jl} dV_g
        \\&=\int_M 3|\nabla\xi|^2-6|d^*\xi|^2+6R_{lk}\xi_{kj}\xi_{lj}-6\langle \mathring{R}(\xi),\xi \rangle dV_g,
    \end{align*}
    so
    \begin{align*}
        \int_M\langle  \overline{L}_f(\xi),\xi \rangle dV_g&=-\frac{1}{6}\| d\xi \|^2-\|d^*\xi\|^2-\|d^*\xi-(d^B)^*\xi\|^2
        \\&=-\frac{1}{2}\|\nabla\xi\|^2-\|d^*\xi-(d^B)^*\xi\|^2-\int_M R_{lk}\xi_{kj}\xi_{lj}+\langle \mathring{R}(\xi),\xi \rangle dV_g
    \end{align*}
In the Bismut-flat case, we apply (\ref{R}) and get
\begin{align*}
    \langle \mathring{R}(\xi),\xi \rangle&=R_{iljk}\xi_{ik}\xi_{lj}
    \\&= \frac{1}{4} (H_{ikm}H_{ljm}-H_{ijm}H_{lkm})\xi_{ik}\xi_{lj}
    \\&=\big(d^*\xi-(d^B)^*\xi\big)^2-\frac{1}{4}H_{ijm}H_{lkm}\xi_{ik}\xi_{lj}.
\end{align*}
On the other hand, the Bianchi identity suggests that 
\begin{align*}
    2\langle \mathring{R}(\xi),\xi \rangle&=-R_{ljik}\xi_{ik}\xi_{lj}
    \\&=\frac{1}{2}H_{lim}H_{jkm}\xi_{ik}\xi_{lj}.
\end{align*}
In conclusion, 
    \begin{align*}
        \big(d^*\xi-(d^B)^*\xi\big)^2=\frac{1}{2}H_{ijm}H_{klm}\xi_{ik}\xi_{jl}=2 \langle \mathring{R}(\xi),\xi \rangle,
    \end{align*}
    and
    \begin{align*}
        \int_M\langle  \overline{L}_f(\xi),\xi \rangle dV_g=-\frac{1}{2}\|\nabla\xi\|^2-\frac{1}{2}\|d^*\xi-(d^B)^*\xi\|^2-\int_M R_{lk}\xi_{kj}\xi_{lj}dV_g.
    \end{align*}
Note that any Bismut-flat manifold has nonnegative Ricci curvature so if $  \overline{L}_f(\xi)=0$, we get $\xi$ is Levi-Civita parallel and $\int_M R_{lk}\xi_{kj}\xi_{lj} dV_g=0$. Moreover,
\begin{align*}
    \int_MR_{lk}\xi_{kj}\xi_{lj} dV_g= \int_M\frac{1}{2}(H_{lab}\xi_{lj})^2 dV_g=0,
\end{align*}
we conclude that $\xi$ is $\nabla^\pm-$parallel. For the other direction, if we assume that $\xi$ is $\nabla^\pm-$parallel, it is clear to see that $ \overline{L}_f(\xi)=0$.

Lastly, we suppose that $\gamma=\xi\circ J\in IGSD$. We note that 
\begin{align}\label{Prc}
    \int_M R_{lk}\xi_{kj}\xi_{lj} dV_g= \int_M R_{lk}\gamma_{kj}\gamma_{lj} dV_g= \int_M R_{lk}\gamma_{jk}\gamma_{jl} dV_g=0.
\end{align}
Thus, 
\begin{align*}
    H_{lab}\gamma_{lj}=H_{lab}\gamma_{jl}=0
\end{align*}
and we deduce that $\gamma$ is $\nabla$-parallel and $\nabla^\pm$-parallel.

\end{proof}

Note that (\ref{Prc}) is true only when $\gamma=0$ in the Hopf surface and Calabi-Eckann surface, we conclude that
\begin{corollary}
    Hopf surface and Calabi-Eckmann surface are rigid in the space of variations fixing the complex structure.
\end{corollary}

\subsection{The case for varying complex structures}

In this section, we consider the general case. First of all, we decompose $\overline{L}_f(\widetilde{\eta})$ into symmetric part and skew-symmetric part. 
\begin{align}\label{Lfeta}
    \overline{L}_f(\widetilde{\eta})&=\frac{1}{2}\Delta_f^B\widetilde{\eta}_{ij}+\frac{1}{2}\big(\mathring{R}^B(\widetilde{\eta})_{ij}+ \mathring{R}^B(\widetilde{\eta})_{ji}\big)-\frac{1}{2}\big(\nabla^B_m\widetilde{\eta}_{kj}H_{mki}+\nabla^B_m\widetilde{\eta}_{ki}H_{mkj}\big)-\frac{1}{4}\big( H^2_{ki}\widetilde{\eta}_{kj}+H^2_{kj}\widetilde{\eta}_{ki} \big)\nonumber
    \\&\kern2em +\frac{1}{2}\big(\mathring{R}^B(\widetilde{\eta})_{ij}-\mathring{R}^B(\widetilde{\eta})_{ji}\big)-\frac{1}{2}\big(\nabla^B_m\widetilde{\eta}_{kj}H_{mki}-\nabla^B_m\widetilde{\eta}_{ki}H_{mkj}\big)-\frac{1}{4}\big( H^2_{ki}\widetilde{\eta}_{kj}-H^2_{kj}\widetilde{\eta}_{ki} \big).
\end{align}

\begin{proposition}
  Let $(M,g,H,J,f)$ be a compact Bismut--Hermitian--Einstein manifold. Define
  \begin{align*}
      D(\alpha)_{ij}\coloneqq \frac{1}{2}\Delta_f^B\alpha_{ij}+\frac{1}{2}\big(\mathring{R}^B(\alpha)_{ij}+ \mathring{R}^B(\alpha)_{ji}\big)-\frac{1}{2}\big(\nabla^B_m\alpha_{kj}H_{mki}+\nabla^B_m\alpha_{ki}H_{mkj}\big)+H_{kim}H_{kjl}\alpha_{ml}, 
  \end{align*}
  where $\alpha$ is a anti-Hermitian, symmetric two tensor. Then,
  \begin{align*}
      \overline{L}_f(\gamma\circ J)=0\Longleftrightarrow d\xi=C,\quad d^*_f\xi=0, \quad D(\widetilde{\eta})=0,
  \end{align*}
  where $\gamma\circ J=-\xi+\widetilde{\eta}\in\ker\overline{\divg}_f$ is given in (\ref{GammaJ}). Therefore, if $\gamma\in IGSD$, then $d\xi=C$, $d^*_f\xi=0$ and $D(\widetilde{\eta})=0$.
\end{proposition}
\begin{proof}
    On a Bismut--Hermitian--Einstein manifold, the second variation of $\lambda$ suggests that
    \begin{align*}
        \frac{d^2}{dt^2}\lambda(\gamma)=\int_M \big\langle \overline{L}_f(\gamma\circ J),\gamma\circ J \big\rangle e^{-f}dV_g=0\Longleftrightarrow  \|d^*_f\xi\|_f^2=\|d\xi-C\|_f^2=0.
    \end{align*}
Take $d\xi=C$ and $d_f^*\xi=(d_f^B)^*\xi=0$, (\ref{Lfxi}) reduces to
\begin{align*}
     \overline{L}_f(\xi)_{ij}&=\frac{1}{2}(\Delta_{d,f}\xi)_{ij}-\frac{1}{4}\big((d\xi)_{mkj}H_{mki}+(d\xi)_{mki}H_{mkj} \big) 
     \\&=-\frac{1}{2}(d^*_fC)_{ij}-\frac{1}{4}(C_{mkj}H_{mki}+C_{mki}H_{mkj} ) 
     \\&=\frac{1}{2}\big( H_{lik}\nabla_l\widetilde{\eta}_{kj}+\nabla_lH_{ijk}\widetilde{\eta}_{kl}+H_{jlk}\nabla_l\widetilde{\eta}_{ki} \big)-\frac{1}{4}\big( H^2_{li}\widetilde{\eta}_{lj}+ H^2_{lj}\widetilde{\eta}_{li}\big)-H_{kim}H_{kjl}\widetilde{\eta}_{ml}.
\end{align*}
By using the Bianchi identity of Bismut curvature (\ref{BianchiBismut}) and note that $\widetilde{\eta}$ is symmetric, we deduce that the skew symmetric part of $ \overline{L}_f(\widetilde{\eta})$ (\ref{Lfeta}) is 
\begin{align*}
   \frac{1}{2}&\big(\mathring{R}^B(\widetilde{\eta})_{ij}-\mathring{R}^B(\widetilde{\eta})_{ji}\big)-\frac{1}{2}\big(\nabla^B_m\widetilde{\eta}_{kj}H_{mki}-\nabla^B_m\widetilde{\eta}_{ki}H_{mkj}\big)-\frac{1}{4}\big( H^2_{ki}\widetilde{\eta}_{kj}-H^2_{kj}\widetilde{\eta}_{ki} \big)
   \\&=\frac{1}{2}\big( R^B_{kijl}-R^B_{kjil}\big)\widetilde{\eta}_{kl} -\frac{1}{2}\big(\nabla_m\widetilde{\eta}_{kj}H_{mki}-\nabla_m\widetilde{\eta}_{ki}H_{mkj}\big)
    \\&=\frac{1}{2} \nabla^B_l H_{kij}\widetilde{\eta}_{kl}-\frac{1}{2}\big(\nabla_m\widetilde{\eta}_{kj}H_{mki}-\nabla_m\widetilde{\eta}_{ki}H_{mkj}\big)
    \\&=\frac{1}{2} \nabla_l H_{kij}\widetilde{\eta}_{kl}-\frac{1}{2}\big(\nabla_m\widetilde{\eta}_{kj}H_{mki}-\nabla_m\widetilde{\eta}_{ki}H_{mkj}\big).
\end{align*}
Thus, the skew-symmetric part of $ \overline{L}_f(\gamma\circ J)=0$ and the remaining term is 
\begin{align*}
    \overline{L}_f(\gamma\circ J)&=-\overline{L}_f(\xi)+\overline{L}_f(\widetilde{\eta})
    \\&=\frac{1}{2}\Delta_f^B\widetilde{\eta}_{ij}+\frac{1}{2}\big(\mathring{R}^B(\widetilde{\eta})_{ij}+ \mathring{R}^B(\widetilde{\eta})_{ji}\big)-\frac{1}{2}\big(\nabla^B_m\widetilde{\eta}_{kj}H_{mki}+\nabla^B_m\widetilde{\eta}_{ki}H_{mkj}\big)+H_{kim}H_{kjl}\widetilde{\eta}_{ml}
    \\&=D(\widetilde{\eta}).
\end{align*}

\end{proof}

\begin{remark}
 In the Bismut-flat case, we note that if $\gamma\in IGSD$, then
 \begin{align*}
        \overline{\nabla}\gamma=0\Longleftrightarrow\overline{L}_f(\gamma)=0\Longleftrightarrow\overline{L}_f(\gamma\circ J) =0  \Longleftrightarrow \overline{\nabla}(\gamma\circ J)=0.
    \end{align*}
Thus, following \Cref{BFG} we get
\begin{align*}
    \nabla_m\widetilde{\eta}_{ij}&=\frac{1}{2} H_{mik}\xi_{kj}+\frac{1}{2}H_{mjk}\xi_{ki} 
    \\\nabla_{m}\xi_{ij}&=\frac{1}{2}H_{mik}\widetilde{\eta}_{kj}-\frac{1}{2}H_{mjk}\widetilde{\eta}_{ki}.
\end{align*}
Then, it is clear that $d\xi=C$, $d^*\xi=0$ and we note that 
\begin{align} \label{flat}
    \mathring{R}^B(\xi)_{ij}+\mathring{R}^B(\xi)_{ji}&=\frac{1}{2}\big( \nabla_iH_{klj}+\nabla_jH_{kli} \big)\xi_{kl}
    \nonumber\\&=-\frac{1}{2}H_{klj}\nabla_i \xi_{kl}-\frac{1}{2}H_{kli}\nabla_j\xi_{kl}
    \nonumber\\&=-\frac{1}{2}H_{kmi}H_{klj}\widetilde{\eta}_{ml},
\end{align}
where we used the fact that $H_{klj}\xi_{kl}=0$ by the assumption. Following (\ref{flat}), one can further deduce that $D(\widetilde{\eta})=0$.  

.
\end{remark}

\section{Local stability of pluriclosed steady soliton}

In this section, we study the stability of pluriclosed flow with fixed complex structure. More precisely, we consider a compact, steady pluriclosed soliton $(M,g_E,J,f_E)$ with $H_E=-d^c\omega_E$ and the solution of pluriclosed flow $(g_t,b_t)$ satisfy
   \begin{align*}
       \frac{\partial}{\partial t}\omega=-\rho_B^{1,1}, \quad \frac{\partial}{\partial t}\beta=-\rho_B^{2,0}.
   \end{align*}
   where $\omega_t=-g_t\circ J$ and $\beta_t=\sqrt{-1}b^{2,0}$. As we know that the pluriclosed flow is gauge equivalent to the generalized Ricci flow (\Cref{PGG}). Thus, we immediately get the following corollary by \cite{K} Theorem 1.8 and \Cref{MTT}.

\begin{corollary}
    Let $(M,g_E,J,f_E)$ be a compact, steady pluriclosed soliton with $H_E=-d^c\omega_E$. Assume 
    \begin{enumerate}
        \item $\lambda(g_E,H_E)$ is a local maximum or
        \item $(g_E,H_E)$ is linearly stable and all infinitsimal solitonic deformation are integrable.
    \end{enumerate}
     Then, for every neighborhood $\mathcal{U}$ of $(g_E,b_E)$, there exists a neighborhood $\mathcal{V}$ satisfies the following property: 
    
    For any pluriclosed flow $(g_t,b_t)$ starting at $(g_0,b_0)\in\mathcal{V}$, there exists a family of automorphism $\{(\varphi_t,B_t)\}\in \GDiff_{H_E}$ such that $(\varphi_t,B_t)\cdot(g_t,b_t)$ stays in $\mathcal{U}$ for all time and converges to some steady pluriclosed soliton $\mathcal{G}_\infty$.
    
\end{corollary}

In the following, we investigate some special cases.

\subsection{Bismut-flat case}

First, we consider the Bismut-flat case. In this special case, we can argue that all infinitsimal solitonic deformations are integrable.

\begin{proposition}\label{IS}
    Suppose $(M,g,H,J)$ is a compact, pluriclosed Bismut-flat manifold and $\gamma=\xi\circ J\in IGSD$ is a general variation where $\xi$ is a 2-form. Then, $\gamma$ is integrable. Moreover, $\gamma$ can induce a family of Bismut-flat metric locally. 
\end{proposition}
\begin{proof}
    Write $\gamma=\xi\circ J=h-K$, where $h$ is the symmetric part of $\gamma$ and $-K$ is skew-symmetric part. Define $(\Tilde{g},\Tilde{b})$ by 
    \begin{align*}
        \Tilde{g}=g+\epsilon h, \quad \Tilde{b}=b+\epsilon K.
    \end{align*}
    W.L.O.G, we can assume that $b=0$ and we can pick $\epsilon$ small enough such that $\Tilde{g}$ is a metric. From \Cref{BFIGSD} it follows that $\gamma$ is $\nabla-$parallel and $\nabla^\pm-$parallel. Thus, both $h$ and $K$ are $\nabla-$parallel. Fix an orthonormal frame $\{e_i\}$ with respect to the metric $g$, we compute
    \begin{align*}
        \widetilde{\Gamma}^B_{ijk}&=\widetilde{\Gamma}_{ijk}+\frac{1}{2}\widetilde{H}_{ijk}
        \\&=\Big( \Gamma_{ijk}+\frac{\epsilon}{2}\big( \partial_i(h_{jk})+\partial_j(h_{ik})-\partial_k(h_{ij})  \big)\Big)+\frac{1}{2}\Big( H_{ijk}+\epsilon (dK)_{ijk}  \Big)
        \\&=\Gamma^B_{ijk}.
    \end{align*}
    Thus, $\widetilde{\nabla}^B=\nabla^B$ and $\gamma$ induces a family of Bismut-flat metrics.
\end{proof}

From the second variation formula \Cref{MT1}, it is clear that all the Bismut-flat metrics are linearly stable. In addition, \Cref{IS} shows that all infinitsimal solitonic deformations are integrable in the plurclosed setting. Thus, we get
\begin{corollary}
Suppose $(M,g,H,J)$ is a compact, pluriclosed Bismut-flat manifold. Then, $(M,g,H,J)$ is pluriclosed flow dynamically stable and the convergence rate is exponential.    
\end{corollary}

\begin{remark}
  In \cite{MJJ}, the authors showed that the pluriclosed flow on Bismut flat metric has long time existence and converges to some Bismut-flat metric. In this paper, we justify this result locally by studying the second variation of $\lambda$.   
\end{remark}

\subsection{General case}

In the general case, we recall the Aeppli cohomology.
\begin{defn} \label{Aeppli}
    Let $(M,J)$ be a complex manifold. Define the \emph{real $(1,1)-$Aeppli cohomology} via
    \begin{align*}
        H^{1,1}_{A,\mathbb{R}}\coloneqq\frac{\Big\{ \ker\sqrt{-1}\partial\overline{\partial}:\Lambda^{1,1}_{\mathbb{R}}\longrightarrow\Lambda^{2,2}_{\mathbb{R}}  \Big\}}{\big\{ \partial \overline{\alpha}+\overline{\partial}\alpha\kern0.3em|\kern0.3em\alpha\in\Lambda^{1,0}   \big\}}.
    \end{align*}
\end{defn}

After fixing the $(1,1)-$Aeppli cohomology, we observe that the second variation is strictly negative so we can prove the following proposition.
\begin{proposition}\label{LMP}
  Let $(M,g_0,b_0,J,f_0)$ be a compact, steady pluriclosed soliton, $H_0=-d^c\omega_0$. There exists a neighborhood $\mathcal{U}$ of $(g_0,b_0)$ such that
  for any $(g,b)\in\mathcal{U}$ satisfying
  \begin{itemize}
      \item $H=-d^c\omega=H_0+db$ and
      \item $[\omega]=[\omega_0]\in H^{1,1}_A$, where $\omega=g\circ J$. 
  \end{itemize} We have
  \begin{align*} 
      \lambda(g,b)\leq \lambda(g_0,b_0), 
  \end{align*}
  and the equality holds if and only if $(g,b)=(g_0,b_0)$. In other words, the steady pluriclosed soliton is rigid if we fix the complex structure and Aeppli class. 
\end{proposition}

\begin{proof}

We first discuss the variation space. W.L.O.G, we may take $b_0=0$ and suppose $b\in \Omega^{2,0+0,2}$. For any $(g,b)$, we consider a linear path $(g_t,b_t)$ connecting $(g_0,b_0)$ to $(g,b)$. Adopt the notation in section 5, we further write  
\begin{align*}
    \omega_t=\omega_0+t\phi, \quad b_t=t\beta,
\end{align*}
where $\omega_t=g_t\circ J$. Assuming that $[\omega]=[\omega_0]\in H^{1,1}_A$, we have $\phi=\overline{\partial}\alpha+\partial \overline{\alpha}$ for some $\alpha\in \Omega^{1,0}$. Note that $H=-d^c\omega=H_0+db$, so 
\begin{align*}
    i\partial\overline{\partial}\alpha-i\overline{\partial}\partial\overline{\alpha}=db.
\end{align*}
Since $\lambda$ depends only on the value of $g$ and $H$, it suffices to consider the variation $\xi=\phi+\beta\circ J=d\alpha$. 

By the Taylor's expansion,
\begin{align*}
    \lambda(g,b)=\lambda(g_0,b_0)+\Big(\frac{d}{dt}\big|_{t=0} \lambda(g_t,b_t) \Big)+\Big(\frac{1}{2}\frac{d^2}{dt^2}\big|_{t=0} \lambda(g_t,b_t) \Big)+R(g_t,b_t),
\end{align*}
where the remainder is given by
\begin{align*}
    R(g,b)=\int_0^1\frac{1}{2}(t-1)^2 \frac{d^3}{dt^3}\big|_{t=0} \lambda(g_t,b_t) dt.
\end{align*}
We note that the first derivative vanishes since $(g_0,b_0)$ is a soliton and the second derivative is given by 
\begin{align*}
    \frac{d^2}{dt^2}\big|_{t=0}\lambda &=
        \int_M\langle  \overline{L}_f(\xi),\xi \rangle e^{-f}dV_g
  =-\frac{1}{6}\| d\xi \|^2_f-2\|d^*_f\xi\|^2_f=-2\|d^*_fd\alpha\|^2_f\leq0.
\end{align*}
We observe that $\frac{d^2}{dt^2}\lambda=0$ if and only if $d\alpha=0$. Thus, the second variation is strictly negative. By shrinking the neighborhood if necessary, we can find a neighborhood such that $\lambda(g,b)\leq\lambda(g_0,b_0)$ and the equality holds only if $(g,b)=(g_0,b_0)$
\end{proof}

To get the dynamical stability, we note that if $\omega_t$ is a solution of the pluriclosed flow then 
\begin{align*}
    [\omega_t]=[\omega_0]-tc_1(M) \in H^{1,1}_A.
\end{align*}

\begin{corollary}\label{Fcor}
     Let $(M,g_E,b_E,J,f_E)$ be a compact, steady pluriclosed soliton, $H_E=-d^c\omega_E$ and the first Chern class $c_1=0\in H^{1,1}_A$. Then, there exists a neighborhood $\mathcal{U}$ of $(g_E,b_E)$ such that if the pluriclosed flow starts at $(g_0,b_0)\in\mathcal{U}$ with $[\omega_E]=[\omega_0]\in H^{1,1}_A$, there exists a family of automorphism $\{(\varphi_t,B_t)\}\in \GDiff_{H_E}$ such that the modified flow $(\varphi_t,B_t)\cdot(g_t,b_t)$ stays in $\mathcal{U}$ for all time and converges exponentially to some steady pluriclosed soliton $(\varphi_\infty,B_\infty)\cdot (g_E,b_E)$.
\end{corollary}

\begin{proof}
  Suppose the first Chern class vanishes, the solution of pluriclosed flow $(g_t,b_t)$ starting at $(g_0,b_0)$ satisfies 
\begin{align*}
    \lambda(g_t,b_t)\leq \lambda(g_E,b_E)
\end{align*}  
by \Cref{LMP}. Since $\lambda$ is diffeomophism invariance, $\lambda\big( (\varphi_t,B_t)\cdot (g_t,b_t) \big)=\lambda(g_t,b_t)$ for all automorphism $(\varphi_t,B_t)$. Then, the proof of \cite{K} Theorem 1.8 follows and the flow $(\varphi_t,B_t)\cdot (g_t,b_t) $ exists for all time and converges to some pluriclosed steady soliton $(\varphi_\infty,B_\infty)\cdot (g_\infty,b_\infty) $. Since $\lambda(g_\infty,b_\infty)=\lambda(g_E,b_E)$ and $[\omega_\infty]=[\omega_E]$, it follows that $(g_\infty,b_\infty)=(g_E,b_E)$. 
\end{proof}

\begin{remark}
  Due to \Cref{Autconv}, we see that the $(-\nabla_{g_t}f_{\mathcal{G}_t},0)$-gauge fixed generalized Ricci flow exists for long time and converges. Thus, the automorphism we picked in \Cref{Fcor} can be
  \begin{align*}
      \varphi_t=\text{the flow generated by $X_t$ and }  B_t=-d\theta_t-i_{X_t}H_0+d(i_{X_t}b_t),
  \end{align*}
  where $X_t=\theta_t^\sharp-\nabla f_t$.
\end{remark}

\begin{remark}
    In the Bismut Hermitian--Einstein case, the first Chern class $c_1=0$.
\end{remark}

\bibliographystyle{plain}
\bibliography{Reference}

\begin{thebibliography}{10}

\bibitem{Ache2012OnTU}
Antonio~G. Ache.
\newblock On the uniqueness of asymptotic limits of the ricci flow.
\newblock {\em arXiv: Differential Geometry}, 2012.

\bibitem{B}
Arthur~L. Besse.
\newblock {\em Einstein manifolds}, volume~10 of {\em Ergebnisse der Mathematik und ihrer Grenzgebiete (3) [Results in Mathematics and Related Areas (3)]}.
\newblock Springer-Verlag, Berlin, 1987.

\bibitem{Cao1985}
Huai-Dong Cao.
\newblock Deformation of {K}ähler matrics to kähler-eisenstein metrics on compact kähler manifolds.
\newblock {\em Inventiones mathematicae}, 81:359--372, 1985.

\bibitem{C}
Huai-Dong Cao and Meng Zhu.
\newblock On second variation of {P}erelman's {R}icci shrinker entropy.
\newblock {\em Math. Ann.}, 353(3):747--763, 2012.

\bibitem{C2}
Huai-Dong Cao and Meng Zhu.
\newblock Linear stability of compact shrinking ricci solitons, 2023.

\bibitem{Aubin}
Pascale Cherrier.
\newblock Equations de monge-amp{\`e}re sur les vari{\'e}t{\'e}s hermitiennes compactes.
\newblock {\em Bulletin Des Sciences Mathematiques}, 111:343--385, 1987.

\bibitem{MR0267604}
David~G. Ebin.
\newblock The manifold of {R}iemannian metrics.
\newblock In {\em Global {A}nalysis ({P}roc. {S}ympos. {P}ure {M}ath., {V}ol. {XV}, {B}erkeley, {C}alif., 1968)}, pages 11--40. Amer. Math. Soc., Providence, R.I., 1970.

\bibitem{MG}
Mario Garcia-Fernandez.
\newblock Ricci flow, killing spinors, and t-duality in generalized geometry.
\newblock {\em Advances in Mathematics}, 350, 11 2016.

\bibitem{MJJ}
Mario Garcia-Fernandez, Joshua Jordan, and Jeffrey Streets.
\newblock Non-kähler calabi-yau geometry and pluriclosed flow.
\newblock {\em Journal de Mathématiques Pures et Appliquées}, 177:329--367, 2023.

\bibitem{GRF}
Mario Garcia-Fernandez and Jeffrey Streets.
\newblock {\em Generalized {R}icci flow}, volume~76 of {\em University Lecture Series}.
\newblock American Mathematical Society, Providence, RI, [2021] \copyright 2021.

\bibitem{Gursky2011RigidityAS}
Matthew~J. Gursky and Jeff Viaclovsky.
\newblock Rigidity and stability of einstein metrics for quadratic curvature functionals.
\newblock {\em arXiv: Differential Geometry}, 2011.

\bibitem{10.4310/jdg/1214436922}
Richard~S. Hamilton.
\newblock {Three-manifolds with positive Ricci curvature}.
\newblock {\em Journal of Differential Geometry}, 17(2):255 -- 306, 1982.

\bibitem{Has2}
Robert Haslhofer.
\newblock Perelman's lambda-functional and the stability of {R}icci-flat metrics.
\newblock {\em Calc. Var. Partial Differential Equations}, 45(3-4):481--504, 2012.

\bibitem{Has1}
Robert Haslhofer and Reto Müller.
\newblock Dynamical stability and instability of ricci-flat metrics.
\newblock {\em Mathematische Annalen}, 360(1-2):547–553, May 2014.

\bibitem{hitchin}
Nigel Hitchin.
\newblock Lectures on generalized geometry, 2010.

\bibitem{String}
S~Ivanov and G~Papadopoulos.
\newblock Vanishing theorems and string backgrounds.
\newblock {\em Classical and Quantum Gravity}, 18(6):1089, mar 2001.

\bibitem{Ko}
N.~Koiso.
\newblock Einstein metrics and complex structures.
\newblock {\em Invent. Math.}, 73(1):71--106, 1983.

\bibitem{Kr}
Klaus Kr{\"o}ncke.
\newblock {\em Stability of Einstein Manifolds}.
\newblock doctoral thesis, Universit{\"a}t Potsdam, 2014.

\bibitem{Kr2}
Klaus Kr\"{o}ncke.
\newblock Stability and instability of {R}icci solitons.
\newblock {\em Calc. Var. Partial Differential Equations}, 53(1-2):265--287, 2015.

\bibitem{KK}
Kuan-Hui Lee.
\newblock Stability and moduli space of generalized ricci solitons, 2023.

\bibitem{K}
Kuan-Hui Lee.
\newblock The stability of generalized ricci solitons.
\newblock {\em The Journal of Geometric Analysis}, 33(9):273, June 2023.

\bibitem{KKK}
Kuan-Hui Lee.
\newblock The stability of non-k\"ahler calabi-yau metrics, 2024.

\bibitem{C3}
Mansour Mehrmohamadi and Asadollah Razavi.
\newblock Commutator formulas for gradient ricci shrinker and their application to linear stability, 2021.

\bibitem{P}
T.~Oliynyk, V.~Suneeta, and E.~Woolgar.
\newblock A gradient flow for worldsheet nonlinear sigma models.
\newblock {\em Nuclear Physics B}, 739(3):441–458, Apr 2006.

\bibitem{Pe}
Grisha Perelman.
\newblock {The Entropy formula for the Ricci flow and its geometric applications}.
\newblock {\em arxiv}, 7 2006.

\bibitem{polchinski_1998}
Joseph Polchinski.
\newblock {\em String Theory}, volume~1 of {\em Cambridge Monographs on Mathematical Physics}.
\newblock Cambridge University Press, 1998.

\bibitem{Rubio_2019}
Roberto Rubio and Carl Tipler.
\newblock The lie group of automorphisms of a courant algebroid and the moduli space of generalized metrics.
\newblock {\em Revista Matemática Iberoamericana}, 36(2):485–536, Dec 2019.

\bibitem{Ses}
Natasa Sesum.
\newblock Linear and dynamical stability of {R}icci-flat metrics.
\newblock {\em Duke Math. J.}, 133(1):1--26, 2006.

\bibitem{J1}
Jeffrey Streets.
\newblock Regularity and expanding entropy for connection {R}icci flow.
\newblock {\em J. Geom. Phys.}, 58(7):900--912, 2008.

\bibitem{J6}
Jeffrey Streets.
\newblock Pluriclosed flow, born-infeld geometry, and rigidity results for generalized kähler manifolds.
\newblock {\em Communications in Partial Differential Equations}, 41(2):318--374, 2016.

\bibitem{STREETS2017506}
Jeffrey Streets.
\newblock Generalized geometry, t-duality, and renormalization group flow.
\newblock {\em Journal of Geometry and Physics}, 114:506--522, 2017.

\bibitem{J5}
Jeffrey Streets.
\newblock Classification of solitons for pluriclosed flow on complex surfaces.
\newblock {\em Mathematische Annalen}, 375, 12 2019.

\bibitem{J7}
Jeffrey Streets.
\newblock {\em Pluriclosed Flow and the Geometrization of Complex Surfaces}, pages 471--510.
\newblock Springer International Publishing, Cham, 2020.

\bibitem{J3}
Jeffrey Streets and Gang Tian.
\newblock {A Parabolic Flow of Pluriclosed Metrics}.
\newblock {\em International Mathematics Research Notices}, 2010(16):3101--3133, 01 2010.

\bibitem{J8}
Jeffrey Streets and Gang Tian.
\newblock {Regularity results for pluriclosed flow}.
\newblock {\em Geometry \& Topology}, 17(4):2389 -- 2429, 2013.

\bibitem{J4}
Jeffrey Streets and Yury Ustinovskiy.
\newblock Classification of generalized kähler-ricci solitons on complex surfaces.
\newblock {\em Communications on Pure and Applied Mathematics}, 74(9):1896--1914, 2021.

\bibitem{Yau}
Shing-Tung Yau.
\newblock On the {R}icci curvature of a compact {K}ähler manifold and the complex {M}onge-{A}mpere equation.
\newblock {\em Communications on Pure and Applied Mathematics}, 31:339--411, 1978.

\end{thebibliography}
\nocite{*}

\end{document}